\documentclass[11pt,a4paper]{article}
\usepackage[utf8]{inputenc}
\usepackage[T1]{fontenc}
\usepackage{amsmath}
\usepackage{amsfonts}
\usepackage{amssymb}
\usepackage{amsthm}
\usepackage{bm}
\usepackage{enumitem}
\usepackage{upgreek}
\usepackage{mathrsfs}
\usepackage[mathscr]{eucal}

\theoremstyle{plain}
\newtheorem{theorem}{Theorem}
\newtheorem{assumption}[theorem]{Assumption}

\newtheorem{proposition}[theorem]{Proposition}
\newtheorem{definition}[theorem]{Definition}

\newtheorem{remark}[theorem]{Remark}
\newtheorem{corollary}[theorem]{Corollary}
\newtheorem{notation}[theorem]{Notation}
\newtheorem*{result*}{Result}
\newtheorem*{remarque*}{Remark}

\newcommand{\e}{\mathrm{e}}

\newcommand{\dr}{\mathrm{d}}

\newcommand{\sausage}{\mathscr{S}}

\newcommand{\dmu}{\mathrm{D}_\mu}

\usepackage[left=2.5cm,right=2.5cm,top=2cm,bottom=2cm]{geometry}
\usepackage{graphicx}
\usepackage{subcaption}
\usepackage{float}
\usepackage{verbatim}
\usepackage[colorlinks=true,
            linkcolor=black,
            urlcolor=black,
            citecolor=black]{hyperref}
\usepackage{cite}
\usepackage{dsfont}
\usepackage[multiple]{footmisc}

\usepackage{todonotes}

\usepackage{authblk}

\title{Brownian particles controlled by their occupation measure}

\author[*]{Loïc Béthencourt}
\author[*]{Rémi Catellier}
\author[$\dagger$]{Etienne Tanré}
	\affil[*]{Université Côte D'Azur, CNRS, LJAD}
	\affil[$\dagger$]{Université Côte D'Azur, Inria, CNRS, LJAD}
\date{April 10, 2024}

\begin{document}

\maketitle

\begin{abstract}
In this article, we study a finite horizon linear-quadratic stochastic control problem for Brownian particles, where the cost functions depend on the state and the occupation measure of the particles. To address this problem, we develop an Itô formula for the flow of occupation measure, which enables us to derive the associated Hamilton-Jacobi-Bellman equation. Then, thanks to a Feynman-Kac formula and the Boué-Dupuis formula \cite{boue1998variational}, we construct an optimal strategy and an optimal trajectory. Finally, we illustrate our result when the cost-function is the volume of the sausage associated to the particles.\\[3pt]
\textit{Keywords:} Stochastic optimal control, Occupation measure, Brownian particles, Calculus on the space of measures, Hamilton-Jacobi-Bellman equations, Boué-Dupuis formula.\\[3pt]
\textit{MSC2020 AMS classification: 93E20, 49J55, 60G57, 49L12}
\end{abstract}

\section{Introduction}
In this article, we consider $n$ random particles $(X_t)_{t\geq0} =(X_t^1, \cdots, X_t^n)_{t\geq0}$ living in $\mathbb{R}^d$ for $d,n \geq1$, and aim at controlling
them  through their occupation 
measure $\mu_t = \sum_{k=1}^n\int_0^t \delta_{X_s^k} \dr s$. The latter belongs to the set of finite measures on $\mathbb{R}^d$ with compact support, which we denote by 
$\mathscr{M}_c^d$. The particles in question are Brownian particles which only interact through the control. More precisely, let \((B^1_t)_{t\geq0}, \cdots, (B^n_t)_{t\geq0}\)
be $n$ independent \(d\)-dimensional Brownian 
motions on some probability space $(\Omega, \mathcal{F}, \mathbb{P})$. We denote \(B = (B^1,\cdots,B^n)\) and consider the (completed) filtration 
$\mathbb{F} = (\mathcal{F}_t)_{t\geq0}$ generated by $(B_t)_{t\geq0}$. For an adapted process $\alpha = (\alpha^1_t, \cdots, \alpha^n_t)_{t\geq0}$, we  consider the particle system 
$(X_t^\alpha)_{t\geq0} = (X_t^{1, \alpha}, \cdots, X_t^{n, \alpha})_{t\geq0}$ as well as its occupation measure $(\mu_t^\alpha)_{t\geq0}$, defined for any $t\geq0$ 
and \(k\in \{1, \cdots, n\} \) by
\[
 X_t^{k,\alpha}  = \int_0^t \alpha^k_s \dr s + B^k_t\quad\text{and} \quad \mu_t^\alpha = \sum_{k=1}^n\int_0^t \delta_{X_s^{k,\alpha}} \dr s.
\]
Let $\mathcal{A}$ denote the class of progressively measurable processes. Consider some finite horizon $T >0$ and two measurable functions $f,g: \mathscr{M}_c^d \times (\mathbb{R}^{d})^n \to \mathbb{R}$. Our aim is to solve the following minimization problem:
\begin{equation}\label{control_problem_intro}
 J(T) = \inf_{\alpha \in \mathcal{A}}\mathbb{E}\Big[g(\mu_T^\alpha, X_T^\alpha) + \int_0^Tf(\mu_s^\alpha, X_s^\alpha) \dr s + \frac{1}{2}\sum_{k=1}^n\int_0^T |\alpha^k_s|^2 \dr s\Big].
\end{equation}
To this end, we will add the occupation measure in the state space, and regard $(\mu_t^\alpha, X_t^\alpha)_{t\geq0}$ as a couple. 
In order to derive the Hamilton-Jacobi-Bellman (HJB) equation associated with our control problem, we  establish an Itô calculus for the flow 
$t\mapsto(\mu_t^\alpha, X_t^\alpha)$. Let us stress that this HJB equation is a PDE on $\mathbb{R}_+ \times \mathscr{M}_c^d \times (\mathbb{R}^{d})^n$. 
The linear-quadratic framework in which we place ourselves is quite convenient, for two main reasons : \textit{(i)} it allows us to classically solve the HJB equation 
\thanks to a \textit{Hopf-Cole} transform and \textit{(ii)} we can rely on the Boué-Dupuis formula \cite{boue1998variational}, which gives an expression 
of \eqref{control_problem_intro} as a $\log$ Laplace transform of some functional on the Wiener space. In particular, we do not have to establish a Dynamic 
Programming Principle and a verification theorem. This framework also enables us to identify the optimal control $\alpha^*$ and to define a process $(\mu_t^*, X_t^*)_{t\in[0,T]}$ 
which is controlled in an optimal way. Although properly stated in Theorem \ref{main_thm}, our main result is, in essence, the following.
\begin{result*}
 Under some regularity assumptions on the functions $f$ and $g$, there exist a continuous function 
 $\alpha^* : [0,T]\times \mathscr{M}_c^d \times (\mathbb{R}^{d})^n \to \mathbb{R}$ and a filtered probability space 
 $(\Omega, \mathcal{G}, (\mathcal{G}_t)_{t\in[0,T]}, \mathbb{Q})$ supporting a $(\mathcal{G}_t)_{t\in[0,T]}$-adapted 
 process $(\mu_t^*, X_t^{1,*}, \cdots, X_t^{n,*})_{t\in[0,T]}$ such that for any $t\in[0,T]$ and \(k\in\{1,\cdots,n\}\),
\[
 X_t^{k,*}  = \int_0^t \alpha^{k,*}(s, \mu_s^*, X_s^{*}) \dr s + B^k_t, \quad \mu_t^* = \sum_{k=1}^n\int_0^t \delta_{X_s^{k, *}} \dr s
\]
where $(B^1_t)_{t\in[0,T]}, \cdots, (B^n_t)_{t\in[0,T]}$ are \(n\) independent $(\mathcal{G}_t)_{t\in[0,T]}$-Brownian motions in \(\mathbb{R}^d\). 
Moreover, it holds that
\[
 J(T) = \mathbb{E}_{\mathbb{Q}}\Big[g(\mu_T^*, X_T^*) + \int_0^Tf(\mu_s^*, X_s^*) \dr s + \frac{1}{2}\sum_{k=1}^n\int_0^T |\alpha^{k,*}(s, \mu_s^*, X_s^*)|^2 \dr s\Big].
\]
\end{result*}

Let us now comment on the related literature. First, we stress that the Itô calculus for occupation measures outlined in this article is not completely novel and 
closely resembles the framework developed in Tissot-Daguette's recent work \cite{tissotdaguette2023occupied}. Yet, this work mostly concerns one-dimensional 
processes with an emphasis on local times for semimartingales and can not be directly applied to our context. We also emphasize that this calculus involves 
functions on the space of measures and functional calculus on $\mathscr{M}_c^d$, which has seen numerous applications in the last two decades with the 
study of stochastic control for McKean-Vlasov equations and Mean Field games, see for instance~\cite{CD1, CD2}. In particular, we will see that the 
\textit{linear functional derivative}, also called the \textit{flat derivative}, plays a crucial role in our analysis. In contrast with the McKean-Vlasov framework 
where $\mu_t$ is deterministic and is the law of $X_t$, the measure $\mu_t$ considered here is random and can be seen as a process with finite variations 
(for any function $\varphi$, $(\langle \mu_t, \varphi \rangle)_{t\geq0}$ has finite variations). As such, the Itô formula obtained in Proposition~\ref{general_ito} 
only differs from the usual one by an additional term featuring a first order derivative.

Since the occupation measure $\mu_t$ is a functional of the whole path $(X_s)_{s\in[0,t]}$, one could also treat the minimisation problem 
\eqref{control_problem_intro} with the recently developed \textit{path-dependent} approach. Initiated by the seminal work of Dupire \cite{dupire}, 
where the author introduces an Itô calculus for functionals on the space of continuous functions, this theory has since seen numerous works on the subject, 
see \cite{fournie, pham_zang, sporito} for some path-dependent control problems and \cite{MR3161485, MR3474470, MR4337713, cosso2023pathdependent} 
for studies of path-dependent HJB equations. However, this approach is often quite involved whereas the approach with occupation measures turns out to be 
fairly simple, with the advantage that one can compute the optimal control explicitly in our case. We refer to Tissot-Daguette \cite{tissotdaguette2023occupied} 
for a link between the calculus for occupation measures and the path-dependent functional calculus.

Let us say a few words on a problem which partly motivates the study of \eqref{control_problem_intro}, namely the modeling of the evolution of filamentous fungus. In the work of Catellier, D'Angelo and Ricci \cite{catellier_angelo}, the branching filaments which form the mycelium are modeled by branching kinetic diffusions that interact with the entire network as well as with their environment. In the work of Toma\v sevi\'c, Bansaye and Véber \cite{MR4516847}, these networks are modeled by multi-type growth fragmentation processes. We would like to add the following postulate to the model: \textit{in a homogeneous environment, the fungal network optimizes the occupied space}. If we see the network not through the trace left by the processes, but through a slight thickening of this trace, we then aim at maximizing the volume of the ``sausage'' of the branching structure, which is a function of the occupation measure. We can treat here a similar problem for particles which do not branch.
This brings us to the main application of our result.

We illustrate our result with a particular example, outlined in Section \ref{sec:5}, corresponding to the case where the function $f = 0$ and the function $g$ is the opposite of the volume of the sausage defined by the particles. In other terms, we control $n$ Brownian motions in order to maximize their volume. Let us recall that, for $\rho > 0$, the $\rho$-sausage of the particles $(X_t^1, \cdots, X_t^n)$ at time $t \geq0$ is the set $A$ defined by
\[
 A = \bigcup_{k=1}^n \bigcup_{s\in[0,t]}\bar{B}(X_s^k, \rho)
\]
where $\bar{B}(X_s^k, \rho)$ is the closed ball of radius $\rho$, centered in $X_s^k$. The function $g$ is then defined by $-m(A)$ where $m$ denotes the Lebesgue measure. If we reformulate this in terms of occupation measure, we see that for any $\nu \in \mathscr{M}_c^d$, we have
\[
 g(\nu) = -m(\{y \in\mathbb{R}^d, \: \dr(y, \mathrm{supp}(\nu)) \leq 1\}).
\]
Unfortunately, this function is not regular enough for us to apply directly our result. 
We exhibit a sequence $(g_\ell)_{\ell \in \mathbb{N}}$ of smooth functions which approximates the function $g$ 
as $\ell\to \infty$. Thanks to our main result, we get for each $\ell \in \mathbb{N}$, $n$ particles $(X_t^{k,*,\ell})_{t\geq0}$ 
controlled in an optimal way. Then, we show that the process $(X_t^{*,\ell})_{t\geq0} = (X_t^{1,*,\ell}, \cdots, X_t^{n,*,\ell})_{t\geq0}$ 
converges in law as $\ell \to \infty$, see Theorem \ref{main_thm_application}. As we shall see, the law of the limiting process 
$(X_t^{*,\infty})_{t\geq0}$ is given by a Gibbs measure with respect to the Wiener measure where the energy is the function $g$. 
As such, it is closely related to the polymer measure (in the case $n=1$), also called \textit{Edward's model}, studied by 
Varadhan \cite{varadhan1969appendix} for $d=2$ and by Westwater \cite{MR853758, MR667754} for $d=3$ 
(see also Bolthausen \cite{MR1240717}). In this model, the law of the polymer chain is given by a Gibbs measure with respect to the 
Wiener measure where the energy is the self-intersection local time. Since maximizing the volume of the sausage more or less amounts 
to penalize the self-intersection local time, we see that these models should be closely related. This limiting process is yet another self-repelling process.

Let us also stress that, by construction, the dynamics of the optimal trajectory $(X_t^*)_{t\in[0,T]} = (X_t^{1,*}, \cdots, X_t^{n,*})_{t\in[0,T]}$ is driven by its own occupation measure. 
Beyond the control problem that interests us here, this kind of dynamics appears in a significant amount of works, starting with the work of Durrett and Rogers~\cite{MR1165516} on the so-called Brownian polymers, and the works of Benaïm, Ledoux and Raimond~\cite{benaim2002self} and Benaïm and Raimond~\cite{benaim_2,benaim_3,benaim_4} on self-interacting diffusions.
 On the other hand, dynamics and control problems driven by probability measures have been at play in many works, for McKean-Vlasov equations or Mean Field games, see~\cite{CD1, CD2}. 
For McKean-Vlasov dynamics driven by renormalized occupation measures and its applications, we refer to the recent works of Du, Jiang and Li~\cite{MR4580925} and Du, Ren, Suciu and Wang~\cite{du2023self}.

\subsection*{Plan of the paper}

The plan of the paper is as follows: in Section \ref{sec:2}, we recall some important notions of calculus in the space of finite measures with compact support. At the end of this section, we derive a chain-rule formula for a flow of occupation-time measures (see Proposition \ref{chain_rule_occupation}) which, as stated, is a generalization of the one obtained in Tissot-Daguette \cite{tissotdaguette2023occupied}.

In Section \ref{sec:3}, we first obtain an Itô formula for the process $(\mu_t, X_t, H_t)_{t\geq0}$ in a general case where $(X_t)_{t\geq0}$ and $(H_t)_{t\geq0}$ are some continuous semimartingales. We then specify this formula when $(X_t)_{t\geq0}$ is a diffusion process and establish a Feynman-Kac formula for the associated PDE on the space $\mathbb{R}_+ \times \mathscr{M}_c^d \times (\mathbb{R}^d)^n$.

In Section \ref{sec:4}, we finally study the control problem \eqref{control_problem_intro} by solving the associated HJB equation and relying extensively on the Boué-Dupuis formula. This leads us to state and prove our main result, see Theorem~\ref{main_thm}.

Finally, in Section \ref{sec:5}, we illustrate our result, controlling Brownian particles by the volume of their sausage, see Theorem~\ref{main_thm_application}.

\section{Calculus on the space of measures}\label{sec:2}

As explained above, we first derive some general results concerning calculus on the space of finite measures. 
In Section~\ref{sec:21}, we start by reminding the basics of spaces of finite measures. 
Section~\ref{sec:22} is devoted to the related calculus.

\subsection{Basics on the space of finite measures}\label{sec:21}
	
	In the following, we  denote by $\mathscr{M}_c^d$ the space of finite measures on $(\mathbb{R}^d, \mathcal{B}(\mathbb{R}^d))$ with compact support. 
	This space is endowed with the weak topology, \textit{i.e.} the weakest topology which makes the 
	maps $\mu \mapsto \int_{\mathbb{R}^d}\varphi \dr \mu =: \langle \mu, \varphi \rangle$ continuous for any $\varphi \in C_b(\mathbb{R}^d)$, 
	where $C_b(\mathbb{R}^d)$ is the space of continuous and bounded functions from $\mathbb{R}^d$ to $\mathbb{R}$. As we shall see, this topology is metrizable.

It is a rather classical consequence of Tychonoff's embedding theorem, see for instance \cite[Th, 17.8]{willard} and \cite[Chap. 1, page 9]{stroock_var}, that if $\rho$ is an equivalent metric on $\mathbb{R}^d$ which makes it a totally bounded space, then the space $U_\rho(\mathbb{R}^d)$ of bounded and uniformly continuous functions for the metric $\rho$ is separable. Let us consider such a metric $\rho$ and a dense sequence $(\varphi_n)_{n\in\mathbb{N}}$ of $U_\rho(\mathbb{R}^d)$ such that the constant function $\bm{1}$ belongs to this sequence. We introduce the distance $\dr$ on $\mathscr{M}_c^d$ defined for any $\mu,\nu \in \mathscr{M}_c^d$ as
\[
 \dr (\mu, \nu) = \sum_{n=0}^\infty 2^{-n}[1 \wedge |\langle \mu, \varphi_n \rangle - \langle \nu, \varphi_n \rangle|].
\]
This metric induces on $\mathscr{M}_c^d$ the weak topology, see for instance \cite[Th. 1.1.2]{stroock_var} with some minor adaptations to the space $\mathscr{M}_c^d$. 

\subsubsection*{Compact subsets of $\mathscr{M}_c^d$}

The proofs in this work require compactness arguments on the spaces $\mathscr{M}_c^d$. Therefore we first describe some compact subsets of this space. 
For $A, B >0$, we introduce the set
\begin{equation}\label{compact_set}
 \mathscr{K}_{A,B} = \{\mu \in \mathscr{M}_c^d, \: \mu(\mathbb{R}^d) \leq A \text{ and }\mathrm{supp}(\mu) \subset [-B, B]^d\}.
\end{equation}
We also introduce for $A > C > 0$ and $B> 0$ the set $\mathscr{K}_{A,B,C} = \{\mu \in \mathscr{K}_{A,B}, \: \mu(\mathbb{R}^d) \geq C\}$. 
To show that $\mathscr{K}_{A,B,C}$ is compact, we mainly paraphrase Ocello \cite[Sec. 4.2]{antonio2023controlled}. 
The set $\mathscr{K}_{A,B,C}$ is homeomorphic to $\mathcal{P}_B \times [C, A]$ where 
$\mathcal{P}_B = \{\mu \in \mathcal{P}(\mathbb{R}^d), \:\mathrm{supp}(\mu) \subset [-B, B]^d\}$ and $\mathcal{P}(\mathbb{R}^d)$ 
is the set of probability measures on $\mathbb{R}^d$. The corresponding homeomorphism is $\mu \mapsto (\mu / \mu(\mathbb{R}^d), \mu(\mathbb{R}^d))$. 
The set $\mathcal{P}_B$ is obviously tight and therefore by Prokhorov's theorem, it is relatively compact. It is also closed and as a consequence, it is compact. 
At this point, we conclude that $\mathscr{K}_{A,B,C}$ is a compact set since it is homeomorphic to a product of compact spaces. 
Let us now prove that $\mathscr{K}_{A,B}$ is also compact. Consider a sequence $(\mu_n)_{n\in\mathbb{N}}$  in $\mathscr{K}_{A,B}$. 
Then either
\begin{enumerate}[label=(\roman*)]
 \item there exists $k_1 \geq 2$ such that $\#\{n \in \mathbb{N}, \: \mu_n \in \mathscr{K}_{A,B,A / {k_1}}\} = \infty$. In this case, we can extract a subsequence which lies in $\mathscr{K}_{A,B,A / {k_1}}$ and since it is a compact space, we can further extract a converging subsequence, whose limit belongs to $\mathscr{K}_{A,B,A / {k_1}}$.
 
 \item for any $k \geq 2$, we have $\#\{n \in \mathbb{N}, \: \mu_n \in \mathscr{K}_{A,B,A / k}\} < \infty$. In this case, we can extract a subsequence $(\mu_{n_k})_{k\geq 2}$ such that for any $k\geq 2$, $\mu_{n_k}(\mathbb{R}^d) < A / k$. Then, this subsequence converges to the null measure, which is a point of $\mathscr{K}_{A,B}$.
\end{enumerate}
This proves that $\mathscr{K}_{A,B}$ is indeed a compact set.

\subsection{Chain-rule formula for the flow of finite measures}\label{sec:22}

Let us define the derivability of functions defined on $\mathscr{M}_c^d$. 
First for $\mu \in \mathscr{M}_c^d$, the space $\mathrm{L}^1(\mu)$ denotes the space of Borel functions which are integrable with respect to $\mu$.

\begin{definition}\label{def_linear_deri}
 A function $u : \mathscr{M}_c^d \to \mathbb{R}$ is said to have a linear functional derivative if it is continuous and if there exists a mapping $\delta_\mu u : \mathscr{M}_c^d\times \mathbb{R}^d \to \mathbb{R}$ satisfying the following properties
 \begin{enumerate}[label=(\roman*)]
  \item \(\delta_\mu u \) is continuous for the product topology.
  \item For any $\mu,\nu \in \mathscr{M}_c^d$, the function 
  $x \mapsto \sup_{\lambda \in [0,1]}|\delta_\mu u (\lambda\nu + (1-\lambda)\mu, x)|$ is in $\mathrm{L}^1(\mu)\cap \mathrm{L}^1(\nu)$.
  \item For any $\mu,\nu \in \mathscr{M}_c^d$, we have
 \begin{equation}\label{def_derivative_mes}
  u(\nu) - u(\mu) = \int_0^1 \int_{\mathbb{R}^d}\delta_\mu u (\lambda\nu + (1-\lambda)\mu, x)(\nu - \mu)(\dr x) \dr \lambda.
 \end{equation}
 \end{enumerate}
The set of such functions is denoted by $\mathrm{C}^1(\mathscr{M}_c^d)$.
\end{definition}

\begin{proof}
 We quickly check here that the quantity $\int_0^1 \int_{\mathbb{R}^d}\delta_\mu u (\lambda\nu + (1-\lambda)\mu, x)(\nu - \mu)(\dr x) \dr \lambda$ is well-defined. 
 First for any $(\mu,x)\in\mathscr{M}_c^d\times\mathbb{R}^d$, the map $\lambda \mapsto \delta_\mu u (\lambda\nu + (1-\lambda)\mu, x)$ is continuous on $[0,1]$ 
 since $\mu \mapsto\delta_\mu u (\mu, x)$ is continuous. Since $\sup_{\lambda \in [0,1]}|\delta_\mu u (\lambda\nu + (1-\lambda)\mu, x)| \in \mathrm{L^1}(\mu)$, 
 it comes that the map $\lambda \mapsto \int_{\mathbb{R}^d}\delta_\mu u (\lambda\nu + (1-\lambda)\mu, x)\nu(\dr x)$ is continuous on $[0,1]$ and thus, is integrable. 
 The same argument applies to the other term.
\end{proof}
\begin{remark}
	In our setting of compact support, the requirement \textit{(ii)} in Definition~\ref{def_linear_deri} can be deduced from \textit{(i)}. However, 
	in the more general context of  finite measures (without compactness assumption), this condition cannot be removed.
Usually, authors assume that the linear derivative is bounded. In our work, we use integration by parts arguments which somehow requires 
instead linear  growth conditions.
Although this extension is quite straightforward, it requires some technicalities that we develop below.
	\end{remark}
 Regarding the uniqueness of the linear derivative, we have the following result, which is more or less known, see for instance Ren and Wang \cite{MR4379563}.
\begin{proposition}
 Let $u\in \mathrm{C}^1(\mathscr{M}_c^d)$, then for any $(\mu, x)\in \mathscr{M}_c^d\times \mathbb{R}^d$, it holds that
 \[
  \frac{u(\mu + \varepsilon \delta_x) - u(\mu)}{\varepsilon} \longrightarrow \delta_\mu u(\mu, x) \quad \text{as }\varepsilon\to0.
 \]
As a consequence, if $u\in \mathrm{C}^1(\mathscr{M}_c^d)$, then its linear derivative is uniquely defined.
\end{proposition}
\begin{proof}
 Let $u\in \mathrm{C}^1(\mathscr{M}_c^d)$, then for any $(\mu, x)\in \mathscr{M}_c^d\times \mathbb{R}^d$, we have
 \begin{align*}
  u(\mu + \varepsilon \delta_x) - u(\mu) & =  \int_0^1 \int_{\mathbb{R}^d}\delta_\mu u (\lambda(\mu + \varepsilon \delta_x) + (1-\lambda)\mu, y)(\mu + \varepsilon \delta_x - \mu)(\dr y) \dr \lambda \\
  & =  \varepsilon \int_0^1 \delta_\mu u (\mu + \lambda\varepsilon \delta_x, x) \dr \lambda.
 \end{align*}
Using the continuity of $\mu \mapsto \delta_\mu u(\mu, x)$, we have the convergence of  $\delta_\mu u (\mu + \lambda\varepsilon \delta_x, x)$ 
to $\delta_\mu u (\mu, x)$ for any $\lambda \in [0,1]$. 
Let $R >0$ be large enough so that $\mathrm{supp}(\mu)\cup\{x\} \subset [-R, R]^d$. Then for any $\lambda, \varepsilon \in(0,1]$, 
$\mu + \lambda\varepsilon \delta_x \in \mathscr{K}_{\mu(\mathbb{R}) + 1,R}$ where we recall  that
$\mathscr{K}_{\mu(\mathbb{R}) + 1,R}$ defined by  \eqref{compact_set} 
is a compact subset of $\mathscr{M}_c^d$. As a consequence, there exists a positive constant $C$ such that for any 
$\lambda, \varepsilon \in(0,1]$, $\delta_\mu u (\mu + \lambda\varepsilon \delta_x, x) \leq C$ and we can apply the dominated convergence theorem, which completes the proof.
\end{proof}

We also have the following chain rule formula.
\begin{proposition}\label{prop:chain_rule_measure}
 Let $n \geq 1$, $u_1, \ldots, u_n$ be functions in $\mathrm{C}^1(\mathscr{M}_c^d)$, $h \in \mathrm{C}^1(\mathbb{R}^n)$ and define the function $g : \mathscr{M}_c^d \to \mathbb{R}$  as $g(\mu) = h(u_1(\mu), \cdots, u_n(\mu))$. Then $g \in \mathrm{C}^1(\mathscr{M}_c^d)$ and for any $(\mu,x) \in \mathscr{M}_c^d \times \mathbb{R}^d$, we have
 \[
  \delta_\mu g(\mu, x) = \sum_{k=1}^n\partial_{x^k}h(g_1(\mu), \cdots, g_n(\mu))\delta_\mu u_k(\mu, x).
 \]
\end{proposition}
The proof of this result can be found in Martini \cite[Prop. 2.20]{martini2023kolmogorov} but it is assumed therein that the linear derivatives are bounded as well as the derivative of $h$. Let us quickly give the additional arguments needed when boundedness is not required.
\begin{proof}
 We only treat the case $n = 1$ and let us first check that items \textit{(i)} and \textit{(ii)} from Definition~\ref{def_linear_deri} are satisfied by $(\mu,x)\mapsto h'(u(\mu))\delta_\mu u(\mu, x)$. The first item is straightforward and regarding the second, we see that for $\mu,\nu \in \mathscr{M}_c^d$ and any $x\in\mathbb{R}^d$,
 \[
  \sup_{\lambda\in[0,1]}|h'(u (\theta_\lambda))\delta_\mu u (\theta_\lambda, x)| \leq \sup_{\lambda\in[0,1]}|h'(u (\theta_\lambda))| \times \sup_{\lambda\in[0,1]}|\delta_\mu u (\theta_\lambda, x)|
 \]
where $\theta_\lambda = \lambda \mu + (1-\lambda)\nu$. Since the map $\lambda \mapsto h'(u (\theta_\lambda))$ is continuous on $[0,1]$, the first term on the right-hand-side is finite, and since $u\in\mathrm{C}^1(\mathscr{M}_c^d)$, $x\mapsto \sup_{\lambda\in[0,1]}|\delta_\mu u (\theta_\lambda, x)|$ is in $\mathrm{L}^1(\mu)\cap \mathrm{L}^1(\nu)$ by definition, it comes that item \textit{(ii)} is satisfied by $(\mu,x)\mapsto h'(u(\mu))\delta_\mu u(\mu, x)$. Finally, for the third item, we paraphrase the proof given in \cite[Prop. 2.20]{martini2023kolmogorov}. For any $\mu,\nu \in \mathscr{M}_c^d$, we have
 \[
  h(u(\nu)) - h(u(\mu)) = \int_0^1\partial_\lambda h(u(\theta_\lambda)) \dr \lambda = \int_0^1h'(u(\theta_\lambda))\partial_\lambda u(\theta_\lambda) \dr \lambda.
 \]
We show that $\partial_\lambda u(\theta_\lambda) = \int_{\mathbb{R}^d}\delta_\mu u(\theta_\lambda, x)(\nu - \mu)(\dr x)$ 
for any $\lambda \in(0,1)$ which will complete the proof. For any $\varepsilon \in (0,1-\lambda)$, we have
\begin{align*}
  \frac{u(\theta_{\lambda + \varepsilon}) - u(\theta_\lambda)}{\varepsilon} = & \int_0^1\int_{\mathbb{R}^d}\delta_\mu u(\theta_\lambda + s\varepsilon(\nu - \mu), x)(\nu - \mu)(\dr x)\dr s \\
  = & \frac{1}{\varepsilon}\int_0^\varepsilon\int_{\mathbb{R}^d}\delta_\mu u(\theta_\lambda + r(\nu - \mu), x)(\nu - \mu)(\dr x) \dr r
\end{align*}
where we remark that $s\theta_{\lambda+\varepsilon} + (1-s)\theta_\lambda = \theta_\lambda + s\varepsilon(\nu - \mu)$ 
and make the change of variable $\varepsilon s = r$. Since $u\in\mathrm{C}^1(\mathscr{M}_c^d)$, the map $r\mapsto \delta_\mu u(\theta_\lambda + r(\nu - \mu), x)$ is continuous on $[0,1]$ for any $x \in \mathbb{R}^d$. Moreover, we see that $\theta_\lambda + r(\nu - \mu) = (\lambda + r)\nu + (1 - (\lambda + r))\mu$ so that for any $\lambda \in (0,1)$, we have
\[
 x\mapsto\sup_{r\in[0,1-\lambda]}|\delta_\mu u(\theta_\lambda + r(\mu - \nu), x)| \in \mathrm{L}^1(\mu) \cap \mathrm{L}^1(\nu)
\]
from which we deduce that for any $\lambda\in(0,1)$, the map $r\mapsto \int_{\mathbb{R}^d}\delta_\mu u(\theta_\lambda + r(\mu - \nu), x)(\mu-\nu)(\dr x)$ is continuous on $[0,1-\lambda]$. Hence we can make $\epsilon\to 0$ and see that for any $\lambda\in(0,1)$, $(u(\theta_{\lambda + \varepsilon}) - u(\theta_\lambda)) / \varepsilon$ converges to $\int_{\mathbb{R}^d}\delta_\mu u(\theta_\lambda, x)(\mu - \nu)(\dr x)$.
\end{proof}

\begin{remark}
 Applying the above result with $n = 2$ and $h(x,y) = xy$, we get that $\delta_\mu[g_1g_2](\mu,x) = g_2(\mu)\delta_\mu g_1(\mu,x) + g_1(\mu) \delta_\mu g_2(\mu,x)$. Note that, if one requires the linear derivative in Definition~\ref{def_linear_deri} to be bounded, then one needs to require the derivatives of $h$ in Proposition~\ref{prop:chain_rule_measure} to be bounded, which is of course not the case for $h(x,y) = xy$. This partly motivates the more general definition of linear derivative given here.
\end{remark}
%
%
In the following proposition, we specify the chain rule (Proposition~\ref{prop:chain_rule_measure}) to
a time dependent flow of occupation measures. This result is fundamental in our work.
It generalizes  \cite[Prop. 3]{tissotdaguette2023occupied}.

\begin{proposition}\label{chain_rule_occupation}
 Let $u \in \mathrm{C}^1(\mathscr{M}_c^d)$, $\nu \in \mathscr{M}_c^d$, $n \geq 1$ and consider $n$ continuous paths $(x_t^k)_{t\geq0}$, $1 \leq k \leq n$, valued in $\mathbb{R}^d$. If we set $\mu_t = \nu + \sum_{k=1}^n \int_0^t\delta_{x_s^k} \dr s$ for any $t\geq0$, then we have
 \[
  u(\mu_t) = u(\nu) +  \sum_{k=1}^n\int_0^t \delta_\mu u (\mu_s, x_s^k) \dr s.
 \]
\end{proposition}

\begin{proof}
 Consider some $t > 0$ and $q \geq 1$ and the uniform subdivision \(0, t/q, \cdots, \ell t/q, \cdots, t\) of \([0,t]\). 
 We classically start by writing
 \begin{align}
  u(\mu_t) - u(\nu) = & \sum_{\ell = 0}^{q-1}\big[u(\mu_{(\ell + 1)t/q}) - u(\mu_{\ell t / q})\big] \nonumber \\
  = & \sum_{\ell = 0}^{q-1}\int_0^1\int_{\mathbb{R}^d}\delta_\mu u(\lambda\mu_{(\ell + 1)t/q} + (1-\lambda)\mu_{\ell t/q}, y)(\mu_{(\ell + 1)t/q} - \mu_{\ell t/q})(\dr y) \dr \lambda \nonumber \\
  = & \sum_{\ell = 0}^{q-1}\sum_{k=1}^n \int_0^1\int_{\ell t / q}^{(\ell + 1)t/q}\delta_\mu u(\lambda\mu_{(\ell + 1)t/q} + (1-\lambda)\mu_{\ell t/q}, x_s^k)\dr s \dr \lambda \label{proof_chain} \\
  = & \sum_{k=1}^n \int_0^t y_s^{k, q} \dr s \nonumber
 \end{align}
where for any $s\in[0, t]$ and any $k \in \{1, \cdots, n\}$, $y_s^{k, q}$ is defined as
\[
 y_s^{k, q} = \int_0^1 \sum_{\ell = 0}^{q-1}\bm{1}_{\{s \in[\ell t/ q, (\ell + 1)t/q)\}}\delta_\mu u(\lambda\mu_{(\ell + 1)t/q} + (1-\lambda)\mu_{\ell t/q}, x_s^k)\dr \lambda.
\]
For the second equality in \eqref{proof_chain}, we used the definition of the linear derivative while in the third, we used that $\mu_{(\ell + 1)t/q} - \mu_{\ell t/q} = \sum_{k=1}^n \int_{\ell t/q}^{(\ell + 1)t/q}\delta_{x_s^k} \dr s$. By the continuity of $(\mu,y) \mapsto \delta_\mu u(\mu, y)$, it is clear that for any $s\in[0, t]$, any $k \in \{1, \cdots, n\}$, and any $\lambda \in[0,1]$, we have 
\[
 \sum_{\ell = 0}^{q-1}\bm{1}_{\{s \in[\ell t/ q, (\ell + 1)t/q)\}}\delta_\mu u(\lambda\mu_{(\ell + 1)t/q} + (1-\lambda)\mu_{\ell t/q}, x_s^k) \longrightarrow \delta_\mu u(\mu_s, x_s^k) \quad \text{as }q\to\infty.
\]
Let $A = nt + \nu(\mathbb{R}^d)$, $B = \max_{k\in\{1, \cdots, n\}}\sup_{s\in[0,t]}|x_s^k|$, and recall the definition \eqref{compact_set} of the compact set $\mathscr{K}_{A,B}$.
 Since $(\mu,y) \mapsto \delta_\mu u(\mu, y)$ is continuous and since $\mathscr{K}_{A,B} \times [-B, B]^d$ is a compact subset of 
 $\mathscr{M}^d_c \times \mathbb{R}^d$, there exists a constant $C > 0$ such that for any $(\mu, y) \in \mathscr{K}_{A,B} \times [-B, B]^d$, $\delta_\mu u(\mu, y) \leq C$.
  Since for any $s\in[0, t]$,  $k \in \{1, \cdots, n\}$,  $\lambda \in[0,1]$ and  
  $\ell \in \{0, \cdots, q-1\}$, $(\lambda\mu_{(\ell + 1)t/q} + (1-\lambda)\mu_{\ell t/q}, x_s^k) \in \mathscr{K}_{A,B} \times [-B, B]^d$, 
\[
 \delta_\mu u(\lambda\mu_{(\ell + 1)t/q} + (1-\lambda)\mu_{\ell t/q}, x_s^k) \leq C \quad \text{and} \quad y_s^{k, q} \leq C.
\]
Hence, we can apply the dominated convergence theorem twice to deduce that for any $k\in\{1, \cdots, n\}$, $\int_0^t y_s^{k, q} \dr s \to \int_0^t \delta_\mu u (\mu_s, x_s^k) \dr s$, which completes the proof.
\end{proof}
\begin{remark}
	The previous proof applies to the more general setting of weighted occupation measures, that is of the form \(\mu_t = \nu + \sum_k a_k \int_0^t\delta_{x^k_s} \dr s\).
\end{remark}
\section{Stochastic calculus and occupation measures}\label{sec:3}

In this Section, we aim at deriving an It\^o type formula for occupation measures of diffusion processes. In Subsection \ref{subsec:general-ito} we prove a general It\^o formula, that we specify in Subsection \ref{subsec:ito-formula-diffusion} for diffusion processes. This allows us to obtain a Feynman-Kac formula (Proposition \ref{prop_feynman}) for a class of PDEs on \(\mathbb{R}_+ \times \mathscr{M}_c^d \times (\mathbb{R}^d)^n\). 

\subsection{A general Itô formula}\label{subsec:general-ito}
In this subsection, we consider a filtered probability space $(\Omega, \mathcal{F}, \mathbb{F}, \mathbb{P})$ supporting two $\mathbb{F}$-adapted semimartingales $(X_t)_{t\geq0}$ and $(H_t)_{t\geq0}$ respectively valued in $(\mathbb{R}^d)^n$ and $\mathbb{R}^p$ for some $d,n,p \geq 1$ and such that $X_0 = x$ and $H_0 = h$ for some $(x,h)\in (\mathbb{R}^d)^n \times \mathbb{R}^p$. The process $(X_t)_{t\geq0}$ represents a system of $n$ particles $(X_t^k)_{t\geq0}$, $k \in \{1, \cdots, n\}$, living in $\mathbb{R}^d$, where we have $X_t = (X_t^1, \cdots, X_t^n)$. Next, we introduce the occupation process $(\mu_t)_{t\geq0}$ of the particles, valued in $\mathscr{M}_c^d$, and defined as
\[
 \mu_t = \nu + \sum_{k=1}^n\int_0^t \delta_{X_s^k} \dr s
\]
where $\nu\in\mathscr{M}_c^d$. Note that for any $A\in \mathcal{B}(\mathbb{R}^d)$, we have $\mu_t(A) = \nu(A) + \sum_{k=1}^n\int_0^t \bm{1}_A(X_s^k) \dr s$. We see that $(\mu_t)_{t\geq0}$ is $\mathbb{F}$-adapted. Let us now give some definitions.
\begin{definition}
 A function $u : \mathscr{M}_c^d \times \mathbb{R}^p \to \mathbb{R}$ is said to be in $\mathrm{C}^{1,2}(\mathscr{M}_c^d \times \mathbb{R}^p)$ if the following assertions hold.
 \begin{enumerate}[label=(\roman*)]
  \item $u$ is jointly continuous for the product topology.
  \item For any $\nu \in\mathscr{M}_c^d$, the function $h \mapsto u(\nu, h)$ is twice differentiable and the functions $(\nu, h) \mapsto \nabla_h u(\nu, h)$ and $(\nu, h) \mapsto \nabla_h^2 u(\nu, h)$ are jointly continuous.
  \item For any $h\in\mathbb{R}^p$, the function $\mu \mapsto u(\mu,h)$ possesses a linear derivative which we denote $\delta_\mu u(\mu, h)(y)$. Moreover, the map $(\mu , h, y) \mapsto \delta_\mu u(\mu, h)(y)$ is jointly continuous.
 \end{enumerate}
\end{definition}
\noindent
We have the following Itô formula.
\begin{proposition}[Itô formula]\label{general_ito}
 Let $u\in\mathrm{C}^{1,2}(\mathscr{M}_c^d \times \mathbb{R}^p)$. Then almost surely for any $t\geq0$,
 \[
  u(\mu_t, H_t) = u(\nu, h) + \int_0^t \nabla_h u(\mu_s, H_s) \dr H_s + \frac{1}{2}\int_0^t \nabla_h^2 u (\mu_s, H_s) \dr \langle H \rangle_s + \sum_{k=1}^n\int_0^t \delta_\mu u(\mu_s, H_s)(X_s^k) \dr s.
 \]
\end{proposition}
The proof is quite similar to the classical proof for It\^o formula with an extra care for the measure term. 
\begin{proof}
Let $t\geq 0$ be fixed, and \(q\geq 1\). We again consider the uniform subdivision of \([0,t]\) with mesh \(t/q\). We  write
\begin{align*}
 u(\mu_t, H_t) - u(\nu, h)  = &\sum_{\ell = 1}^{q - 1}\left[u(\mu_{(\ell+1)\frac{t}{q}}, H_{(\ell+1)\frac{t}{q}}) - u(\mu_{\ell\frac{t}{q}}, H_{(\ell+1)\frac{t}{q}})\right] \\
& +  \sum_{\ell = 1}^{q - 1}\left[u(\mu_{\ell\frac{t}{q}}, H_{(\ell+1)\frac{t}{q}}) - u(\mu_{\ell\frac{t}{q}}, H_{\ell\frac{t}{q}})\right] \\
   =: &\bm{\mathrm{I}}_q + \bm{\mathrm{II}}_q.
\end{align*}
The second term $\bm{\mathrm{II}}_q$  converges to 
$\int_0^t \nabla_h u(\mu_s, H_s) \dr H_s + \frac{1}{2}\int_0^t \nabla_h^2 (\mu_s, H_s) \dr \langle H \rangle_s$ as $q\to\infty$.
 We omit the proof as it is the same proof as the one of the classical Itô formula: perform a Taylor expansion of $u$ in the variable $h$, 
 as the first variable is frozen and then let $q\to\infty$. Regarding $\bm{\mathrm{I}}_p$, we have thanks to Proposition~\ref{chain_rule_occupation} that
\[
 \bm{\mathrm{I}}_q = \sum_{k=1}^n\int_0^t Y_s^{k,q} \dr s \quad \text{where} 
 \quad Y_s^{k,q} = \sum_{\ell = 0}^{q - 1}\delta_\mu u(\mu_s, H_{(\ell+1)\frac{t}{q}})(X_s^k)\bm{1}_{ (\ell\frac{t}{q}, (\ell+1)\frac{t}{q}]}(s).
\]
Clearly, almost surely for any $s\in[0,t]$, $Y_s^{k,q} \to \delta_\mu u(\mu_s, H_s)(X_s^k)$ as $q\to\infty$. 
Moreover, if we let $X_t^* = \sup_{s\in[0,t]}|X_s|$ and $H_t^* = \sup_{s\in[0,t]}|H_s|$, then for any $s_1, s_2\in[0,t]$ and any $k\in\{1, \cdots, n\}$, we have
\[
 (\mu_{s_1}, H_{s_2}, X_{s_1}^k) \in \mathscr{K}_{nt + \nu(\mathbb{R}^d), X_t^*} \times [-H_t^*, H_t^*]^p \times [-X_t^*, X_t^*]^d 
\]
where we recall that $\mathscr{K}_{A,B} = \{\mu \in \mathscr{M}_c^d, \: \mu(\mathbb{R}^d) \leq A \text{ and }\mathrm{supp}(\mu) \subset [-B, B]^d\}$. 
Since the set $\mathscr{K}_{nt+ \nu(\mathbb{R}^d), X_t^*} \times [-H_t^*, H_t^*]^p \times [-X_t^*, X_t^*]^d$ is a compact subset 
of $\mathscr{M}_c^d\times \mathbb{R}^p \times \mathbb{R}^d$ and since $(\mu, h, x) \mapsto\delta_\mu u(\mu, h)(x)$ is jointly continuous, 
there exists a random constant $C$ such that almost surely, for any $s\in[0,t]$ and any $k\in\{1, \cdots, n\}$, $|Y_s^{k,q}| \leq C$. 
By the dominated convergence theorem, we get that almost surely $\bm{\mathrm{I}}_q$ converges to 
$\sum_{k=1}^n\int_0^t \delta_\mu u(\mu_s, H_s)(X_s^k) \dr s$ as $q\to\infty$, which establishes the result.
\end{proof}
\begin{remark}\label{remark_ito}
As usual, if some coordinates  of $(H_t)_{t\geq0}$ have finite variations, we do not need to assume that $u$ is $\mathrm{C}^2$ in these coordinates, 
we can assume that it is only $\mathrm{C}^1$ without restriction.
\end{remark}

\subsection{Related formulas for diffusion processes}\label{subsec:ito-formula-diffusion}

In this subsection, we finally consider the process $(\mu_t, X_t)_{t\geq0}$ valued in $\mathrm{E} := \mathscr{M}_c^d \times (\mathbb{R}^{d})^n$ and defined by the following equation: for \(k\in \{1,\cdots,n\}\),
\begin{equation}\label{eds_occ}
X_t^k  = x^k + \int_0^t b^k(X^k_s) \dr s + \int_0^t \sigma^k (X^k_s) \dr B^k_s\quad \text{ and }\quad  \mu_t = \nu + \sum_{k=1}^n\int_0^t \delta_{X_s^k} \dr s
\end{equation}
where $(B_t)_{t\ge 0}=(B_t^1,\cdots,B_t^n)_{t\geq0}$ is an $d\times n$-dimensional Brownian motion and $(\nu,x) \in \mathrm{E}$. 
Again, $X_t = (X_t^1, \cdots, X_t^n)$ represents the system of $n$ particles living in $\mathbb{R}^d$ and \(x = (x^1,\cdots,x^n)\in (\mathbb{R}^d)^n\). 
We assume that $b^k:\mathbb{R}^{d} \to \mathbb{R}^{d}$ and $\sigma^k:\mathbb{R}^{d} \to \mathbb{R}^{d\times d}$ 
are smooth enough to ensure
 pathwise uniqueness for the SDE satisfied by $(X_t)_{t\geq0}$. 
 We denote by $\mathcal{L}$ the generator of the full vector \(X\) defined for functions $f \in C^2((\mathbb{R}^{d})^n)$ as
 \begin{equation}\label{generator}
  \mathcal{L}f = b \cdot \nabla f + \frac{1}{2}\mathrm{tr}(\nabla^2 f \sigma \sigma^{T}),
 \end{equation}
 where $b= (b^1, \cdots, b^n)$ and $\sigma$ is the block diagonal matrix formed with the matrices $\sigma^k$, $k\in\{1, \cdots, n\}$. For a function $u:\mathbb{R}_+ \times \mathrm{E} \to \mathbb{R}$, we say that $u \in \mathrm{C}^{1,1,2}(\mathbb{R}_+ \times \mathrm{E})$ if it is differentiable in the variables 
 $t, \nu$, twice differentiable in the variable $x$ and all the derivatives are jointly continuous. We have the following formula which is a direct corollary of Proposition~\ref{general_ito} 
 and Remark~\ref{remark_ito}.
\begin{corollary}\label{ito_special_case}
 Let $u \in \mathrm{C}^{1,1,2}(\mathbb{R}_+ \times \mathrm{E})$. Then almost surely, for any $t\geq0$,
 \begin{align*}
  u(t, \mu_t, X_t) = & u(0, \nu, x) + \int_0^t\Big[\partial_t u + \mathcal{L} u\Big](s, \mu_s, X_s) \dr s \\
   & + \sum_{k=1}^n\int_0^t \delta_\mu u(s, \mu_s, X_s)(X_s^k) \dr s + \sum_{k=1}^{n}\int_0^t \nabla_{x^k} u(s, \mu_s, X_s) \cdot \sigma^k(X_s^k) \dr B_s^k.
 \end{align*}
\end{corollary}
\noindent
In the following, we will denote by $\mathbb{P}_{(\nu, x)}$ the law of the process $(\mu_t, X_t)_{t\geq0}$ solution of \eqref{eds_occ} started at $(\nu, x)\in\mathrm{E}$. We also need the following definitions.
\begin{definition}
 We say that a function 
 $h : \mathrm{E}\to \mathbb{R}$ is $\mathrm{C}_{F}(\mathrm{E})$ if the following assertions hold.
 \begin{enumerate}[label=(\roman*)]
  \item The function $h$ is twice differentiable in the variable $x$ and $\nabla_x h$ possesses a linear derivative $\delta_\mu\nabla_x h$ such that for all 
  $(\nu, x) \in \mathrm{E}$, $y\mapsto \delta_\mu\nabla_x h(\nu, x)(y)$ is differentiable.
  
  \item The function $h$ possesses a linear derivative  such that $(x,y) \mapsto \delta_\mu h(\nu, x)(y)$ is twice differentiable, and $\nu \mapsto\delta_\mu h(\nu, x)(y)$ possesses a linear derivative such that $(y,y') \mapsto\delta_{\mu\mu}^2 h(\nu, x)(y, y')$ is also twice differentiable.
  
  \item All the derivatives are bounded and jointly continuous.
 \end{enumerate}
\end{definition}
\begin{definition}
 We say that a function $h: \mathrm{E}\to \mathbb{R}$ is $\mathrm{L}(B)$ if for any $A, B, M, T > 0$, it holds that
 \[
  \mathbb{E}\Big[\sup_{(t,\nu, x) \in \mathrm{K}}|h(\nu + \theta_t^x, x + B_t)|\Big] < \infty
 \]
where $\mathrm{K} = [0,T]\times \mathscr{K}_{A, B} \times [-M, M]^{nd}$, $(B_t)_{t\geq0} = (B_t^1, \cdots, B_t^n)_{t\geq0}$ is an $nd$ dimensional Brownian motion, $x = (x^1, \cdots, x^n)$ and $\theta_t^x = \sum_{k=1}^n \int_0^t\delta_{x^k + B_s^k} \dr s$.
\end{definition}
We also introduce the following notation.
\begin{notation}
 For a function $u\in\mathrm{C}^{1,1,2}(\mathbb{R}_+\times \mathrm{E})$, we denote by $\dmu u$ the function from $\mathbb{R}_+\times \mathrm{E}$ to $\mathbb{R}$ defined for any $(t, \nu, x)$ as $\dmu u(t, \nu, x) = \sum_{k=1}^n \delta_\mu u(t, \nu, x)(x^k)$.
\end{notation}
\noindent
We have the following Feynman-Kac formula.
\begin{proposition}[Feynman-Kac formula]\label{prop_feynman}
 Assume that the coefficient $\sigma$ is bounded. Let $f \in \mathrm{C}^0(\mathrm{E})$ bounded from below, $h \in \mathrm{C}^0(\mathrm{E})$ and let $u \in \mathrm{C}^{1,1,2}(\mathbb{R}_+\times \mathrm{E})$ such that $\nabla_x u$ is bounded. If $u$ is a solution to the PDE
 \begin{equation}\label{feynman_pde}
    \left\lbrace
        \begin{aligned}
            \partial_t& u(t,\nu,x)  + f(\nu,x)u(t,\nu,x) = [\mathcal{L} u + \dmu u](t,\nu, x) & \: \: t >0, (\nu, x) \in \mathrm{E} \\
            u&(0,\nu, x) = h(\nu,x) & \: \: (\nu, x)\in\mathrm{E}
        \end{aligned}
    \right.
\end{equation}
where $\mathcal{L}$ is given by \eqref{generator}, then, for any $T\geq0$ and any $(\nu, x) \in \mathrm{E}$, we have
\begin{equation}\label{feynman_sol}
 u(T,\nu, x) = \mathbb{E}_{(\nu, x)}\Big[\exp\Big(-\int_0^Tf(\mu_s, X_s)\dr s\Big)h(\mu_T, X_T)\Big]
\end{equation}
where $(\mu_t, X_t)_{t\geq0}$ is solution of \eqref{eds_occ}.
Moreover whenever $f\in \mathrm{C}_F(\mathrm{E})$, $h \in\mathrm{C}_F(\mathrm{E})\cap \mathrm{L}(B)$, $b^k = 0$ and $\sigma^k = Id$, the function defined by \eqref{feynman_sol} is $\mathrm{C}^{1,1,2}(\mathbb{R}_+\times \mathrm{E})$ and is the unique classical solution to~\eqref{feynman_pde}.
\end{proposition}

\begin{proof}
Let us prove the first statement. Consider some $T > 0$ fixed and, applying Corollary~\ref{ito_special_case} with the function $(t, \nu, x) \mapsto u(T-t, \nu, x)$, we get for any $t\in[0,T]$
\[
 u(T-t, \mu_t, X_t) = u(T,\nu, x) + \int_0^t\Big[-\partial_t u + \mathcal{L} u + \dmu u\Big](T -s, \mu_s, X_s) \dr s + M_t
\]
where $(M_t)_{t\geq0}$ is a local martingale (given in Corollary~\ref{ito_special_case}). Then, applying again Itô's formula to the process $\exp(-\int_0^tf(\mu_s, X_s)\dr s)u(T-t, \mu_t, X_t)$, we get
\begin{align*}
 \e^{-\int_0^tf(\mu_s, X_s)\dr s}u(T-t, \mu_t, X_t) = & u(T,\nu, x) + \int_0^t\e^{-\int_0^sf(\mu_r, X_r)\dr r}\dr M_s  \\
  &  + \int_0^t\e^{-\int_0^sf(\mu_r, X_r)\dr r}[-\partial_t u - fu + \mathcal{L} u + \dmu u](T -s, \mu_s, X_s) \dr s \\
  = & u(T,\nu, x) + \int_0^t\e^{-\int_0^sf(\mu_r, X_r)\dr r}\dr M_s
\end{align*}
where we used that $u$ is a solution to \eqref{feynman_pde}. Since $f$ is bounded from below and $\sigma$ and $\nabla_x u$ are bounded the local martingale $\int_0^t\e^{-\int_0^sf(\mu_r, X_r)\dr r}\dr M_s$ is a true martingale and we can pass to the expectation and deduce that
\[
 \mathbb{E}_{(\nu, x)}\Big[\exp\Big(-\int_0^tf(\mu_s, X_s)\dr s\Big)u(T-t, \mu_t, X_t)\Big] = u(T,\nu,x).
\]
Letting $t = T$ shows the result.

We now move to the second part of the statement and will therefore assume that $f,h \in \mathrm{C}_F(\mathrm{E})$ and that the coefficients are given by $b = 0$ and $\sigma = 1$. In this case, the function $u$ is such that for any $t\geq0$ and any $(\mu, x) \in \mathrm{E}$,
\[
 u(t,\nu, x) = \mathbb{E}\bigg[\exp\Big(-\int_0^tf(\nu + \theta_s^{x}, x + B_s)\dr s\Big)h\big(\nu + \theta_t^{x}, x + B_t\big)\bigg]
\]
where $(B_t)_{t\geq0} = ((B_t^1, \cdots, B_t^n))_{t\geq0}$ is an $nd$ dimensional Brownian motion starting at $0$ and 
$\theta_t^{x} = \sum_{k=1}^n\int_0^t\delta_{x^k + B_s^k} \dr s$. The function $u$ is well-defined since $f$ is bounded from below 
and $h \in \mathrm{L}(B)$. Let us show that it is also jointly continuous. First it is clear that a.s., 
the function $(t,\nu, x) \mapsto \e^{-\int_0^tf(\nu + \theta_s^{x}, x + B_s)\dr s}h(\nu + \theta_t^{x}, x + B_t)$ is jointly continuous. Next, we consider some fixed $(t,\nu, x) \in \mathbb{R}_+ \times \mathrm{E}$, we let $R > 0$ such that $\mathrm{supp}(\nu) \subset [-R, R]^d$ and we introduce the set 
\[
 \mathrm{K} = [0, t+1] \times \mathscr{K}_{\nu(\mathbb{R}^d), R} \times [-2|x|, 2|x|]^{nd}.
\]
Since $h \in \mathrm{L}(B)$, the quantity 
$\mathbb{E}[\sup_{(t,\nu, x) \in \mathrm{K}}|h(\nu + \theta_t^x, x + B_t)|]$ is finite and we deduce that $u$ is continuous at $(t,\nu, x)$. Since this holds for any $(t,\nu, x) \in \mathbb{R}_+ \times \mathrm{E}$, $u$ is jointly continuous.
This argument will be applied many times below to show that the various derivatives of $u$ are jointly continuous. 

\bigskip\noindent
\textit{Step 1:} we show that for any $(t,x) \in \mathbb{R}_+ \times(\mathbb{R}^d)^n $, $\nu \mapsto u(t,\nu, x)$ possesses a linear derivative $\delta_\mu u(t,\nu, x)(y)$ which is continuous in all the variables. First, we fix $(t,x)\in \mathbb{R}_+ \times (\mathbb{R}^d)^n$ and we show that a.s., $\nu \mapsto \int_0^tf(\nu + \theta_s^{x}, x + B_s)\dr s$ possesses a linear derivative. For any $\nu_1, \nu_2 \in \mathscr{M}_c^d$, we have
\begin{multline*}
 \int_0^tf(\nu_1 + \theta_s^{x}, x + B_s)\dr s - \int_0^tf(\nu_2 + \theta_s^{x}, x + B_s)\dr s \\
 = \int_0^t\int_0^1\int_{\mathbb{R}^d}\delta_\mu f(\theta_s^x + \lambda \nu_1 + (1-\lambda)\nu_2, x + B_s)(y)(\nu_1 - \nu_2)(\dr y) \dr \lambda \dr s.
\end{multline*}
Since $f \in \mathrm{C}_F(\mathrm{E})$, its linear derivative is bounded and since $\nu_1$ and $\nu_2$ have a finite mass, we can exchange the above integrals and check that a.s., 
the linear derivative of $\nu \mapsto \int_0^tf(\nu + \theta_s^{x}, x + B_s)\dr s$ is $\int_0^t\delta_\mu f(\nu + \theta_s^x, x+B_s)(y) \dr s$. 
We now define the random function $\Phi : \mathbb{R}_+ \times \mathrm{E} \to \mathbb{R}$ by
\begin{equation}
\Phi(t, \nu, x) = \exp\left(-\int_0^tf(\nu + \theta_s^{x}, x + B_s)\dr s\right)h(\nu + \theta_t^{x}, x + B_t).
\end{equation}
 We  apply Proposition~\ref{prop:chain_rule_measure} and deduce that $\Phi$ possesses a linear derivative and  for any $(\nu, y) \in \mathrm{E}$
\[
 \delta_\mu \Phi(t, \nu, x)(y) =\delta_\mu h(\nu + \theta_t^{x}, x + B_t)(y)\e^{-\int_0^tf(\nu + \theta_s^{x}, x + B_s)\dr s} -\Phi(t, \nu, x)\int_0^t\delta_\mu f(\nu + \theta_t^x, x+B_s)(y) \dr s.
\]
Reminding that $h \in \mathrm{L}(B)$, $\delta_\mu f$ and $\delta_\mu h$ are bounded and $f$ is bounded from below, we obtain  $\mathbb{E}[|\delta_\mu \Phi(t,\nu,x)(y)|]$ is finite. Finally, for any $\nu_1, \nu_2 \in \mathscr{M}_c^d$, we have 
\[
u(t,\nu_1, x) - u(t,\nu_2, x) =  \mathbb{E}\bigg[\int_0^1\int_{\mathbb{R}^d}\delta_\mu \Phi(t, \lambda \nu_1 + (1-\lambda)\nu_2, x)(y)(\nu_1 - \nu_2)(\dr y) \dr \lambda\bigg].
\]
Since $\nu_1$ and $\nu_2$ have finite mass, we can exchange the above integrals which shows that $\nu \mapsto u(t,\nu, x)$ possesses a linear derivative: 
$\delta_\mu u(t,\nu, x)(y) = \mathbb{E}[\delta_\mu \Phi(t,\nu,x)(y)]$. Since the linear derivatives of $f$ and $h$ are bounded, since $f$ is bounded from below 
 and since $h\in\mathrm{L}(B)$, this derivative is jointly continuous.

\bigskip\noindent
\textit{Step 2:} we show that for  $(t,\nu) \in \mathbb{R}_+ \times \mathscr{M}_c^d$, the function $x \mapsto u(t,\nu, x)$ is twice differentiable. 
Without loss of generality, we write the proof in the case  $d =1$ and $n\geq1$. 
We consider some $\varepsilon > 0$ and for $k\in\{1,\cdots, n\}$, we denote by $\varepsilon_k$ the vector $\varepsilon e_k$ where $e_k$ is the usual $k$th unit vector. We start by writing for any $(t,\nu, x)\in \mathbb{R}_+ \times \mathrm{E}$ that 
\begin{align*}
 \Phi(t,\nu, x+\varepsilon_k) - \Phi(t,\nu, x) = & \Phi(t,\nu, x+\varepsilon_k) - \e^{-\int_0^tf(\nu + \theta_s^{x+\varepsilon_k}, x  + B_s)\dr s}h(\nu + \theta_t^{x+\varepsilon_k}, x + B_t) \\
  & + \e^{-\int_0^tf(\nu + \theta_s^{x+\varepsilon_k}, x  + B_s)\dr s}h(\nu + \theta_t^{x+\varepsilon_k}, x + B_t) - \Phi(t,\nu,x) \\
  = & \bm{\mathrm{I}}_{\varepsilon, k}(t,\nu,x) + \bm{\mathrm{II}}_{\varepsilon, k}(t,\nu,x).
\end{align*}
We start with \( \bm{\mathrm{I}}_{\varepsilon, k} \). In order to simplify the notations, we decouple in the function $\Phi$ the initial conditions with respect to the occupation measure and to the Brownian motion.
We therefore introduce for any $(t,\nu, x, z) \in \mathbb{R}_+ \times \mathrm{E}\times (\mathbb{R}^d)^n$ the random function
\[
 \Lambda(t, \nu, x, z) := \exp(-\int_0^tf(\nu + \theta_s^{z}, x + B_s)\dr s)h(\nu + \theta_t^{z}, x + B_t)
\]
It should be clear that $\Lambda$ is differentiable in the  variable $x$ and that for any $k\in\{1,\cdots, n\}$,
\[
 \partial_{x^k} \Lambda(t, \nu, x, z) = \partial_{x^k} h(\nu + \theta_t^z, x + B_t)\e^{-\int_0^tf(\nu + \theta_s^{z}, x+ B_s)\dr s} - \Lambda(t, \nu, x, z)\int_0^t \partial_{x^k} f(\nu + \theta_s^z, x+B_s) \dr s.
\]
Since $z \mapsto \theta_s^z$ is continuous, we see that $\partial_x \Lambda$ is jointly continuous. Moreover, we have
\[
 \bm{\mathrm{I}}_{\varepsilon, k}(t,\nu,x) = \int_0^\varepsilon \partial_{x^k} \Lambda(t, \nu, x + u, x + \varepsilon) \dr u.
\]
It is clear that $\lim_{\varepsilon\to 0} \bm{\mathrm{I}}_{\varepsilon, k}(t,\nu,x) / \varepsilon  = \partial_{x^k} \Lambda(t, \nu, x, x)$. 

We proceed in a similar manner for \( 
\bm{\mathrm{II}}_{\varepsilon, k} \). 
First, we consider for any $(t,\nu, x, z)$ the random function $\ell(t, \nu, x, z) := h(\nu + \theta_t^z, x + B_t)$. Then for any $\varepsilon >0$, we have
\[
 \ell(t, \nu, x, z + \varepsilon_k) - \ell(t, \nu, x, z) = \int_0^1\int_{\mathbb{R}^d}\delta_\mu h(\nu + \lambda\theta_t^{z+\varepsilon_k} + (1-\lambda)\theta_t^{z}, x + B_t)(y)(\theta_t^{z+\varepsilon_k} - \theta_t^{z})(\dr y) \dr \lambda.
\]
Recalling that $\theta_t^{z+\varepsilon_k}(\dr y) = \int_0^t\delta_{z^k + \varepsilon + B_s^k}(\dr y) \dr s + \sum_{j\neq k}\int_0^t\delta_{z^j + B_s^j}(\dr y) \dr s$, the above quantity is equal to 
\begin{align*}
 \int_0^1 \int_0^t\delta_\mu h(\nu + &\lambda\theta_t^{z+\varepsilon_k} + (1-\lambda)\theta_t^{z}, x + B_t)(z^k + \varepsilon + B_s^k) \dr s \dr \lambda \\
  & - \int_0^1 \int_0^t\delta_\mu h(\nu + \lambda\theta_t^{z+\varepsilon} + (1-\lambda)\theta_t^{z}, x + B_t)(z^k + B_s^k) \dr s \dr \lambda.
\end{align*}
Using the joint continuity and the boundedness of $\delta_\mu h$, it should be clear that the above quantity, divided by $\varepsilon$, converges to
\[
 \int_0^t \partial_y\delta_\mu h(\nu + \theta_t^{z}, x + B_t)(z^k + B_s^k) \dr s.
\]
This shows that the function $\ell$ is differentiable in the variable $z$ and that its derivative is jointly continuous. Now we introduce $m(t,\nu, x, z) := \int_0^tf(\nu + \theta_s^z, x+ B_s) \dr s$. For any $\varepsilon > 0$, we have
\[
 m(t,\nu, x, z+ \varepsilon_k) - m(t,\nu, x, z) = \int_0^t\int_0^1\int_{\mathbb{R}^d}\delta_\mu f(\nu + \lambda\theta_s^{z+\varepsilon_k} + (1-\lambda)\theta_t^{z}, x + B_s)(y)(\theta_s^{z+\varepsilon_k} - \theta_s^{z})(\dr y) \dr \lambda
\]
which is in turn equal to
\begin{align*}
 \int_0^t\int_0^1 \int_0^s\delta_\mu f(\nu + &\lambda\theta_s^{z+\varepsilon_k} + (1-\lambda)\theta_s^{z}, x + B_s)(z^k + \varepsilon + B_u^k) \dr u \dr \lambda \dr s \\
  & - \int_0^t\int_0^1 \int_0^s\delta_\mu f(\nu + \lambda\theta_s^{z+\varepsilon_k} + (1-\lambda)\theta_s^{z}, x + B_s)(z^k + B_u^k) \dr u \dr \lambda \dr s.
\end{align*}
Again, using the joint continuity and boundedness of $\delta_\mu f$, we see that the above quantity dividing by $\varepsilon$ converges to
\[
 \int_0^t\int_0^s\partial_y\delta_\mu f(\nu + \theta_s^{z}, x + B_s)(z^k + B_u^k) \dr u \dr s
\]
as $\varepsilon\to0$. This shows that $m$ is differentiable in the variable $z$ and that its derivative is jointly continuous. All in all, this shows that the random function $\Lambda(t,\nu, x, z)$ is differentiable in the variable $z$ and that its derivative is jointly continuous in the variable $z$. Moreover, we have
\begin{align*}
 \partial_{z^k}\Lambda(t, \nu, x, z) = & \e^{-\int_0^tf(\nu + \theta_s^{z}, x+ B_s)\dr s}\int_0^t \partial_y\delta_\mu h(\nu + \theta_t^{z}, x + B_t)(z^k + B_s^k) \dr s \\
 &- \Lambda(t, \nu, x, z)\int_0^t\int_0^s\partial_y\delta_\mu f(\nu + \theta_s^{z}, x + B_s)(z^k + B_u^k) \dr u \dr s.
\end{align*}
Finally, we see that
\[
\bm{\mathrm{II}}_{\varepsilon, k}(t,\nu,x) = \int_0^\varepsilon \partial_{z^k} \Lambda(t, \nu, x, x + u) \dr u
\]
from which we deduce that $\bm{\mathrm{II}}_{\varepsilon, k}(t,\nu,x) / \varepsilon$ converges to $\partial_{z^k} \Lambda(t, \nu, x, x)$ as $\varepsilon\to0$. This shows that the function $\Phi(t,\nu, x)$ is differentiable in the variable $x$ and that for any $k\in\{1, \cdots,n\}$
\begin{align}
 \partial_{x^k} \Phi(t,\nu, x) = & \partial_{x^k} h(\nu + \theta_t^x, x + B_t)\e^{-\int_0^tf(\nu + \theta_s^{x}, x+ B_s)\dr s} - \Phi(t, \nu, x)\int_0^t \partial_{x^k} f(\nu + \theta_s^x, x+B_s) \dr s \nonumber \\
  & +\e^{-\int_0^tf(\nu + \theta_s^{x}, x+ B_s)\dr s}\int_0^t \partial_y\delta_\mu h(\nu + \theta_t^{x}, x + B_t)(x^k + B_s^k) \dr s \label{gradient_u} \\
  & - \Phi(t, \nu, x)\int_0^t\int_0^s\partial_y\delta_\mu f(\nu + \theta_s^{x}, x + B_s)(x^k + B_u^k) \dr u \dr s. \nonumber
\end{align}
Since all the derivatives of $h$ and $f$ are bounded, $f$ is bounded from below and $h \in \mathrm{L}(B)$, 
we see that $\mathbb{E}[|\partial_x \Phi(t,\nu, x)|]$ is finite and therefore $\partial_x u(t,\nu, x) = \mathbb{E}[\partial_x \Phi(t,\nu, x)]$, which is by the usual arguments continuous. Regarding the second derivative, we can repeat the above arguments to show that $u$ is indeed twice differentiable in the variable $x$. We see that in all generality, the second derivative would involve 21 terms and as the proof is already involved, we only treat the case where $f = 0$. Let us fix some $k \in \{1, \cdots, n\}$ and set $j_k(t, \nu, x) = \partial_{x^k}h(\nu + \theta_t^x, x + B_t)$, we see that we can repeat the arguments of this step to deduce that for any $i\in\{1, \cdots, n\}$
\[
 \partial_{x^i}j_k(t, \nu, x) = \partial_{x^ix^k}h(\nu + \theta_t^x, x + B_t) + \int_0^t\partial_y\delta_\mu\partial_{x^k} h(\nu + \theta_t^{x}, x + B_t)(x^i + B_s^i) \dr s.
\]
Regarding the function $p_k(t,\nu, x) = \int_0^t \partial_y\delta_\mu h(\nu + \theta_t^{x}, x + B_t)(x^k + B_s^k) \dr s$, we can also apply the above arguments and see that for any $i\in\{1, \cdots, n\}$
\begin{align*}
 \partial_{x^i}p_k(t, \nu, x) = &\bm{1}_{\{i = k\}}\int_0^t \partial_{yy}\delta_\mu h(\nu + \theta_t^{x}, x + B_t)(x^k + B_s^k) \dr s \\
  & + \int_0^t \partial_{x^i}\partial_y\delta_\mu h(\nu + \theta_t^{x}, x + B_t)(x^k + B_s^k) \dr s \\
  & + \int_0^t\int_0^s \partial_{y'}\delta_{\mu}\partial_{y}\delta_\mu h(\nu + \theta_t^{x}, x + B_t)(x^k + B_s^k, x^i + B_u^i) \dr u \dr s.
\end{align*}
Repeating the usual arguments involving the boundedness of the derivatives of $f$ and $h$ as well as the fact that $f$ is bounded from below and $h\in\mathrm{L}(B)$, we see that $u$ is twice differentiable and that $\partial_{xx} u(t, \nu, x) = \mathbb{E}[\partial_{x}j(t, \nu, x) + \partial_{x}p(t, \nu, x)] = \mathbb{E}[\partial_{xx}\Phi(t, \nu, x)]$.

\bigskip\noindent
\textit{Step 3:} Finally, for any $(\nu, x) \in \mathrm{E}$, we show that $t\mapsto u(t,\nu, x)$ is differentiable and that is satisfies \eqref{feynman_pde}. For any $\varepsilon > 0$, we have
\begin{align*}
 u(t+\varepsilon,\nu,x) - u(t,\nu,x) = & \mathbb{E}_{(\nu, x)}\Big[\e^{-\int_\varepsilon^{t+\varepsilon}f(\mu_s, X_s)\dr s}h(\mu_{t+\varepsilon}, X_{t+\varepsilon})(\e^{-\int_0^\varepsilon f(\mu_s, X_s)\dr s} - 1)\Big] \\
  & + \mathbb{E}_{(\nu,x)}\Big[\e^{-\int_\varepsilon^{t+\varepsilon}f(\mu_s, X_s)\dr s}h(\mu_{t+\varepsilon}, X_{t+\varepsilon})\Big] - u(t,\nu,x) \\
  = &: \bm{\mathrm{I}}_\varepsilon(t,\nu,x) + \bm{\mathrm{II}}_\varepsilon(t,\nu,x)
\end{align*}
Regarding the first term, it is clear that $\varepsilon^{-1}\bm{\mathrm{I}}_\varepsilon(t,\nu,x)$ converges to 
\[
 - \mathbb{E}_{(\nu, x)}\Big[\e^{-\int_0^{t}f(\mu_s, X_s)\dr s}h(\mu_{t}, X_{t})f(\mu_0, X_0)\Big] = -f(\nu, x)u(t,\nu,x).
\]
Regarding the second, we have by the strong Markov property that
\[
 \bm{\mathrm{II}}_\varepsilon(t,\nu,x) = \mathbb{E}_{(\nu,x)}\Big[u(t, \mu_\varepsilon, X_\varepsilon)\Big] - u(t,\nu, x).
\]
Remember here that $t$ is fixed. By the first steps, the function $(\gamma,z)\mapsto u(t, \gamma, z) \in \mathrm{C}^{1,2}(\mathrm{E})$ and therefore we can apply Itô's formula. For $p\geq 1$ and $\tau_p = \inf\{t \geq 0, \sup_{s\in[0,t]}|X_s| \geq p\}$, we have
\[
 u(t, \mu_{\varepsilon\wedge \tau_p}, X_{\varepsilon\wedge \tau_p}) = u(t, \nu, x) + \int_0^{\varepsilon\wedge \tau_p}[\mathcal{L}u + \dmu u](t, \mu_s, X_s) \dr s + \int_0^{\varepsilon\wedge \tau_p} \nabla_x u(t, \mu_s, X_s) \dr B_s.
\]
Let $M > 0$ such that $\mathrm{supp}(\nu) \subset [-M, M]^d$, $a = \nu(\mathbb{R}^d)$ and recall that for $A, B>0$, the set 
$\mathscr{K}_{A, B} = \{\mu \in \mathscr{M}_c^d, \: \mu(\mathbb{R}^d) \leq A \text{ and }\mathrm{supp}(\mu) \subset [-B, B]^d\}$ 
is a compact subset of $\mathscr{M}_c^d$. Since $(\gamma, z) \mapsto\nabla_x u(t,\gamma, z)$ is jointly continuous, there exists a 
positive constant $C$ such that for any $(\gamma, z) \in \mathscr{K}_{n\varepsilon + a, p + M} \times [-p, p]^{nd}$, $|\nabla_x u(t,\nu, x)| \leq C$. 
Since for any $s\in[0, \varepsilon \wedge \tau_p]$, $(\mu_s, X_s) \in \mathscr{K}_{n\varepsilon + a, p + M} \times [-p, p]^{nd}$, we conclude that 
$(\int_0^{r\wedge\varepsilon\wedge \tau_p} \nabla_x u(t, \mu_s, X_s) \dr B_s)_{r\geq0}$ is a  martingale, so that
\[
 \mathbb{E}_{(\nu,x)}\Big[u(t, \mu_{\varepsilon\wedge \tau_p}, X_{\varepsilon\wedge \tau_p})\Big] = u(t, \nu, x) + \mathbb{E}_{(\nu,x)}\Big[\int_0^{\varepsilon\wedge \tau_p}[\mathcal{L}u + \dmu u](t, \mu_s, X_s) \dr s\Big].
\]
Since $f$ is bounded from above and $h\in\mathrm{L}(B)$, it holds that
\[
 \mathbb{E}_{(\nu,x)}\Big[\sup_{p\in\mathbb{N}}|u(t, \mu_{\varepsilon\wedge \tau_p}, X_{\varepsilon\wedge \tau_p})|\Big] \leq \mathbb{E}_{(\nu,x)}\Big[\sup_{s\in[0,\varepsilon]}u(t, \mu_{s}, X_{s})\Big] < \infty
\]
Therefore, we can let $p\to\infty$ and deduce that
\[
 \bm{\mathrm{II}}_\varepsilon(t,\nu,x) = \mathbb{E}_{(\nu,x)}\Big[\int_0^{\varepsilon}[\mathcal{L}u + \dmu u](t, \mu_s, X_s) \dr s\Big].
\]
Finally, we see that $\varepsilon^{-1}\bm{\mathrm{II}}_\varepsilon(t,\nu,x)$ converges to $\mathcal{L}u(t,\nu, x) + \dmu u(t,\nu, x)$ which completes the proof.
\end{proof}

\begin{remark}
 We emphasize that the previous results, in particular Proposition \ref{prop_feynman}, does not depend on the diagonal structure of the generator $\mathcal{L}$ and could be extended to much more general coefficients $b$ and $\sigma$, possibly depending on the time, the occupation measure and the full space.
\end{remark}

\section{The control problem}\label{sec:4}

Let $(\Omega, \mathcal{F}, \mathbb{P})$ be a probability space supporting some $n\times d$-dimensional Brownian motion $(W_t)_{t\geq0}$, where $n,d \geq 1$. Let also $\mathbb{F} = (\mathcal{F}_t)_{t\geq0}$ be the filtration generated by $(W_t)_{t\geq0}$ after the usual completions. Let $\alpha = (\alpha_t)_{t\geq0}$ be some progressively measurable process \textit{w.r.t.} $\mathbb{F}$. For $(\nu, x) \in \mathrm{E}$, we introduce the process $(\mu_t^{\alpha}, X_t^{\alpha})_{t\geq0}$ valued in $\mathrm{E}$, starting at $(\nu, x)$ defined for any $t\geq0$ by
\begin{equation}\label{triplet_equation}
X_t^{k,\alpha}  = x^k + \int_0^t \alpha^k_s \dr s + W^k_t, \quad \mu_t^\alpha = \nu + \sum_{k=1}^n\int_0^t \delta_{X_s^{k,\alpha}} \dr s
\end{equation}
and $(X_t^\alpha)_{t\geq0} = (X_t^{1, \alpha}, \cdots, X_t^{n, \alpha})_{t\geq0}$. Consider now a finite horizon $T >0$ and let $\mathcal{A}$ denotes the class of progressively measurable processes. Consider also two measurable functions $f,g: \mathrm{E} \to \mathbb{R}$. Our aim is to solve, for any $(\nu, x)  \in \mathrm{E}$, the following minimization problem:
\begin{equation}\label{control_problem}
 \inf_{\alpha \in \mathcal{A}}\mathbb{E}_{(\nu, x)}\Big[g(\mu_T^\alpha, X_T^\alpha) + \int_0^Tf(\mu_s^\alpha, X_s^\alpha) \dr s + \frac{1}{2}\sum_{k=0}^n\int_0^T |\alpha^k_s|^2 \dr s\Big].
\end{equation}
\begin{remark}
In this work, we do not consider the general stochastic optimal control problem with occupation measures, in particular the associated theory of viscosity solutions and the verification theorems.
 Our specific choice of the control structure (linear / quadratic) allows us to consider HJB equations in a simple setting.
 We rely on the Boué-Dupuis formula and a Hopf-Cole transform to solve the HJB equation, to identify the value function and the optimal strategy. This gives us an explicit expression for the law of an optimal trajectory.
\end{remark}

\paragraph{The Boué-Dupuis formula.} To solve this minimization problem, we will first solve the Hamilton-Jacobi-Bellman equation associated with \eqref{control_problem} and then define a process which is controlled in an optimal way. To show that the corresponding control is indeed optimal, we will rely on the Boué-Dupuis formula introduced by Boué and Dupuis \cite{boue1998variational}, see also Budhiraja \cite{budhiraja2024some} for a recent review of the extensions of this formula.
Let us denote by $\mathcal{C}_T$ the space of continuous functions from $[0,T]$ to $(\mathbb{R}^d)^n$ endowed with the uniform distance. The formula says that for any $\mathrm{F}:\mathcal{C}_T \to \mathbb{R}$ which is bounded from above, we have
\begin{equation}\label{boue_dupuis}
 -\log\mathbb{E}\Big[\e^{-\mathrm{F}((W_t)_{t\in[0,T]})}\Big] = \inf_{\alpha \in \mathcal{A}}\mathbb{E}\Big[\mathrm{F}\Big(\big(W_t + \int_0^t \alpha_s \dr s\big)_{t\in[0,T]}\Big) + \frac{1}{2}\sum_{k=0}^n\int_0^T |\alpha^k_s|^2 \dr s\Big]
\end{equation}
This equality was later extended to functionals $\mathrm{F}$ which are not necessarily bounded from above, see for instance \cite{ustunel}. For our purposes, we will assume the following on the functions $f$ and $g$.
\begin{assumption}\label{assump_f_g}
 The function $g: \mathrm{E}\to \mathbb{R}$ is bounded from above and the function $f: \mathrm{E}\to \mathbb{R}$ is bounded.
\end{assumption}
\noindent
We can now apply \eqref{boue_dupuis} to our context. In the following, we always identify $(e_t)_{t\in[0,T]} \in \mathcal{C}_T$ as $(e_t^1, \cdots, e_t^n)_{t\in[0,T]}$. For a fixed $z = (\nu, x)\in\mathrm{E}$, we define the function $\Psi_z : \mathcal{C}_T \to \mathrm{E}$ by
\[
 \Psi_z((e_t)_{t\in[0,T]}) = g\bigg(\nu + \sum_{k=1}^n\int_0^T \delta_{(x + e_s^k)} \dr s, \: x + e_T\bigg) + \int_0^Tf\bigg(\nu + \sum_{k=1}^n\int_0^s \delta_{(x + e_u)} \dr u, \: x + e_s\bigg) \dr s
\]
which is measurable (it is even continuous). By Assumption \ref{assump_f_g}, it is bounded from below and we can therefore apply \eqref{boue_dupuis} with the function $\Psi_z$, which tells us that
\begin{equation*}
 -\log\mathbb{E}\Big[\e^{-\Psi_z((W_t)_{t\in[0,T]})}\Big]  = \inf_{\alpha \in \mathcal{A}}\mathbb{E}\Big[\Psi_z\Big((W_t +\int_0^t \alpha_s \dr s)_{t\in[0,T]}\Big)+ \frac{1}{2}\sum_{k=0}^n\int_0^T |\alpha^k_s|^2 \dr s\Big],
 \end{equation*}
 or in a more explicit way in our context,
 \begin{equation}\label{eq:boue-dupuis-value-function}
  -\log\mathbb{E}\Big[\e^{-\Psi_z((W_t)_{t\in[0,T]})}\Big]  = \inf_{\alpha \in \mathcal{A}}\mathbb{E}_{(\nu, x)}\Big[g(\mu_T^\alpha, X_T^\alpha) + \int_0^Tf(\mu_s^\alpha, X_s^\alpha) \dr s + \frac{1}{2}\sum_{k=0}^n\int_0^T |\alpha^k_s|^2 \dr s\Big].
\end{equation}
However, this formula does not give any information on a minimizing control $\alpha^*$, which will be our main objective in the rest of this note.

\paragraph{The HJB equation.} Let us now solve the HJB equation associated with \eqref{control_problem} by deriving a semi-explicit solution. As explain above, this is only possible because we place ourselves in the linear / quadratic framework. This solution is a Hopf-Cole transform of the solution to the associated Heat equation. Let us  define for any $(t, \nu,x)\in[0,T] \times \mathrm{E}$ the function
\begin{equation}\label{def_u_value}
  u(t, \nu, x) = \mathbb{E}\bigg[\exp\Big(-g(\nu + \theta_{T-t}^x, x+ B_{T-t}) -\int_0^{T-t}f(\nu + \theta_s^x, x+ B_s) \dr s\Big)\bigg]
\end{equation}
where $(B_t)_{t\geq0}$ is an $n\times d$-dimensional Brownian motion and $\theta_t^x = \sum_{k=1}^n\int_0^t \delta_{x^k + B_s^k} \dr s$. We also define the function
\begin{equation}\label{def_c_value}
  c(t, \nu, x) = -\log(u(t, \nu, x)).
\end{equation}
We have the following result.
\begin{proposition}
 Let $f\in \mathrm{C}_F(\mathrm{E})$ and $g \in\mathrm{C}_F(\mathrm{E})$ such that $\exp(-g) \in \mathrm{L}(B)$. The function $c$ defined by \eqref{def_c_value} is a $\mathrm{C}^{1,1,2}(\mathbb{R}_+\times \mathrm{E})$ function which is a solution to the following HJB equation
 \begin{equation}\label{hjb_eq}
    \left\lbrace
        \begin{aligned}
            \partial_t& c(t,\nu,x) + \Big[\frac{1}{2}\Delta_x c + \dmu c\Big](t,\nu, x) + f(\nu, x) = \frac{1}{2}|\nabla_x c|^2(t, \nu, x) & \: \: t \in[0,T), (\nu, x) \in \mathrm{E} \\
            c&(T,\nu, x) = g(\nu, x, v) &\: \: (\nu, x)\in\mathrm{E}
        \end{aligned}
    \right. .
\end{equation}
\end{proposition}

\begin{proof}
 Since $g$ is a $\mathrm{C}_F(\mathrm{E})$ function and the exponential is smooth, the function $h(\nu,x) = \exp(-g(\nu,x))$ is also a $\mathrm{C}_F(\mathrm{E})$ function, see for instance Proposition~\ref{prop:chain_rule_measure}. Since $h \in \mathrm{L}(B)$, we can apply Proposition~\ref{prop_feynman}, which tells us that the function
 \[
  v(t,\nu, x) = \mathbb{E}\bigg[\exp\Big(-\int_0^tf(\nu + \theta_s^x, x+ B_s) \dr s-g(\nu + \theta_{t}^x, x+ B_{t})\Big)\bigg]
 \]
is a $\mathrm{C}^{1,1,2}(\mathbb{R}_+\times \mathrm{E})$ function which is a solution 
 \begin{equation*}
    \left\lbrace
        \begin{aligned}
            \partial_t& v(t,\nu,x) + f(\nu, x)v(t,\nu, x) = \Big[\frac{1}{2}\Delta_x v + \dmu v\Big](t,\nu,x)& \: \: t >0, (\nu,x)\in \mathrm{E} \\
            v&(0,\nu, x) = \exp(-g(\nu, x)) & \: \: (\nu,x)\in \mathrm{E}.
        \end{aligned}
    \right.
\end{equation*}
Of course, for any $t\in[0,T]$, we have $u(t, x, \nu) = v(T-t, x, \nu)$ and we see that $v$ is a solution to
\begin{equation*}
    \left\lbrace
        \begin{aligned}
            \partial_t& u(t,\nu,x) + \Big[\frac{1}{2}\Delta_x u + \dmu u\Big](t,\nu,x)  = f(\nu,x)u(t,\nu, x)& \: \: t\in[0,T), (\nu,x)\in \mathrm{E} \\
            u&(T,\nu, x) = \exp(-g(\nu, x)) & \: \: (\nu,x)\in \mathrm{E}.
        \end{aligned}
    \right.
\end{equation*}
Since $u$ is a $\mathrm{C}^{1,1,2}(\mathbb{R}_+\times \mathrm{E})$ function which is positive, $c$ is also a $\mathrm{C}^{1,1,2}(\mathbb{R}_+\times \mathrm{E})$ function and we have $u = \exp(-c)$ so that
\[
 \partial_t u = -u \partial_t c, \quad \dmu u = -u \dmu c \quad \text{and} \quad \Delta_x u = u(|\nabla_x|^2 c - \Delta_x c).
\]
Inserting these relations into the above PDE completes the proof.
\end{proof}
Let us emphasize that our application of the Boué-Dupuis formula given in \eqref{eq:boue-dupuis-value-function} gives us directly that for $c$ defined by \eqref{def_c_value} we have
\[
 c(0,\nu, x) = \inf_{\alpha \in \mathcal{A}}\mathbb{E}_{(\nu, x)}\Big[g(\mu_T^\alpha, X_T^\alpha) + \int_0^Tf(\mu_s^\alpha, X_s^\alpha) \dr s +
 \frac{1}{2}\sum_{k=0}^n\int_0^T |\alpha^k_s|^2 \dr s
\Big].
\]
We can now define the optimal control, and we set for any $(t, \nu, x) \in \mathbb{R}_+ \times \mathrm{E}$ and any $k\in\{1, \cdots, n\}$, 
\begin{equation}\label{eq:def-alpha-optim}
\alpha^{k,*}(t, \nu, x) = -\nabla_{x^k} c(t, \nu, x).
\end{equation}
We have the following theorem.
\begin{theorem}\label{main_thm}
Let $f,g : \mathrm{E} \to \mathbb{R}$ satisfying Assumption \ref{assump_f_g} such that $f\in \mathrm{C}_F(\mathrm{E})$, $g \in\mathrm{C}_F(\mathrm{E})$ and $\exp(-g) \in \mathrm{L}(B)$. Let also $(\nu, x) \in \mathrm{E}$ be fixed. There exists a filtered probability space $(\Omega, \mathcal{G}, (\mathcal{G}_t)_{t\in[0,T]}, \mathbb{Q})$ supporting a $(\mathcal{G}_t)_{t\in[0,T]}$-adapted process $(\mu_t^*, X_t^*)_{t\in[0,T]}$ such that for any $t\in[0,T]$ and any $k\in\{1, \cdots, n\}$,
\begin{equation}\label{optimal_process}
 X_t^{k,*}  = x + \int_0^t \alpha^{k,*}(s, \mu_s^*, X_s^*) \dr s + W_t^k, \quad \mu_t^* = \nu + \sum_{k=1}^n\int_0^t \delta_{X_s^{k,*}} \dr s
\end{equation}
where $(W_t^1)_{t\in[0,T]}, \cdots, (W_t^n)_{t\in[0,T]}$ are 
\(n\) independent
$(\mathcal{G}_t)_{t\in[0,T]}$-Brownian motions in $\mathbb{R}^d$ and $\alpha^{k,*}$ is defined in Equation \eqref{eq:def-alpha-optim}. Moreover, for any bounded and measurable functional $\mathrm{F}:\mathcal{C}_T \to \mathbb{R}$,
\begin{multline}\label{eq:uniqueness-law-optimal-trajectories}
 \mathbb{E}_{\mathbb{Q}}[\mathrm{F}((X_t^*)_{t\in[0,T]})] \\ = \mathcal{Z}^{-1}\mathbb{E}\left[\mathrm{F}((x + B_t)_{t\in[0,T]})\exp\left(-\int_0^Tf(\nu + \theta_s^x, x+ B_s) \dr s-g(\nu + \theta_{T}^x, x+ B_{T})\right)\right]
\end{multline}
where $\mathcal{Z} = \mathbb{E}[\e^{-\int_0^Tf(\nu + \theta_s^x, x+ B_s) \dr s-g(\nu + \theta_{T}^x, x+ B_{T})}]$. Finally, it holds that
\[
 c(0,\nu, x) = \mathbb{E}_{\mathbb{Q}}\Big[g(\mu_T^*, X_T^*) + \int_0^Tf(\mu_s^*, X_s^*) \dr s + \frac{1}{2}\sum_{k=1}^n\int_0^T |\alpha^{k,*}(s, \mu_s^*, X_s^*)|^2 \dr s\Big].
\]
\end{theorem}

\begin{remark}
Remark that Equation \eqref{eq:uniqueness-law-optimal-trajectories} gives us directly uniqueness of the law for optimal trajectories. Using a Girsanov transform one could show uniqueness in law for the associated SDE given in Equation \eqref{optimal_process}.
\end{remark}

\begin{proof}
 \textit{Step 1:} Let us show that the coefficient $\alpha^*$ is bounded. First, we see that $\alpha^{k,*} = \nabla_{x^k} u / u$ where we recall that
 \[
  u(T-t, \nu, x) = \mathbb{E}\bigg[\exp\Big(-\int_0^tf(\nu + \theta_s^{x}, x + B_s)\dr s\Big)h\big(\nu + \theta_t^{x}, x + B_t\big)\bigg],
 \]
$h = \exp(-g)$, $(B_t)_{t\geq0}$ is a Brownian motion and $\theta_t^{x} = \sum_{k=1}^n\int_0^t\delta_{x^k + B_s^k} \dr s$. 
Let us also recall from \eqref{gradient_u} that for any $k\in\{1, \cdots, n\}$, we have
\begin{align*}
 \nabla_{x^k} u(T-t, \nu, x) = & \mathbb{E}\Big[\nabla_{x^k} h (\nu + \theta_t^{x}, x + B_t)\e^{-\int_0^tf(\nu + \theta_s^{x}, x+ B_s)\dr s}\Big] \\
  & - \mathbb{E}\Big[\int_0^t \nabla_{x^k} f(\nu + \theta_s^x, x+B_s) \dr s \times \e^{-\int_0^tf(\nu + \theta_s^{x}, x+ B_s)\dr s}h(\nu + \theta_t^{x}, x + B_t)\Big] \\
  & + \mathbb{E}\Big[\int_0^t \nabla_y\delta_\mu h(\nu + \theta_t^{x}, x + B_t)(x^k + B_s^k) \dr s \times\e^{-\int_0^tf(\nu + \theta_s^{x}, x+ B_s)\dr s} \Big] \\
  & - \mathbb{E}\Big[\int_0^t\int_0^s\nabla_y\delta_\mu f(\nu + \theta_s^{x}, x + B_s)(x^k + B_u^k) \dr u \dr s \\
   & \qquad \qquad \qquad \qquad \qquad \qquad \qquad \quad \times\e^{-\int_0^tf(\nu + \theta_s^{x}, x+ B_s)\dr s}h(\nu + \theta_t^{x}, x + B_t)\Big].
\end{align*}
Since $h= \exp(-g)$, we have $\nabla_{x^k} h = - h\nabla_{x^k} g$ and $\nabla_y\delta_\mu h = -h \nabla_y\delta_\mu g$ and as a consequence, we have for any $k\in\{1, \cdots, n\}$
\[
 \nabla_{x^k} u(T-t, \nu, x) = - \mathbb{E}\Big[G(t,\nu, x)\e^{-\int_0^tf(\nu + \theta_s^{x}, x+ B_s)\dr s}h(\nu + \theta_t^{x}, x + B_t)\Big]
\]
where 
\begin{align*}
 G(t,\nu,x) = & \nabla_{x^k} g (\nu + \theta_t^{x}, x + B_t) + \int_0^t \nabla_y\delta_\mu g(\nu + \theta_t^{x}, x + B_t)(x^k + B_s^k) \dr s \\
  & + \int_0^t \nabla_{x^k} f(\nu + \theta_s^x, x+B_s) + \int_0^t\int_0^s\nabla_y\delta_\mu f(\nu + \theta_s^{x}, x + B_s)(x^k + B_u^k) \dr u \dr s.  \\
\end{align*}
By assumption, all the derivatives of $g$ and $f$ are bounded and therefore, we see that there is a positive constant $C_T$ such that $|\nabla_{x^k} u| \leq C_T u$ on $[0,T]\times \mathrm{E}$ which shows that $\alpha^{k,*}$ is bounded on $[0,T]\times \mathrm{E}$.

\noindent
\textit{Step 2:} Consider a filtered probability space $(\Omega, \mathcal{G}, (\mathcal{G}_t)_{t\geq0}, \mathbb{P})$ supporting some
independent Brownian motions 
$(B^1_t)_{t\geq0}, \cdots, (B^n_t)_{t\geq0}$ and let $(\nu, x) \in \mathrm{E}$ be fixed. We define for any $t\in[0,T]$, the process $\mu_t^* = \nu + \sum_{k=1}^n\int_0^t \delta_{x^k + B_s^k} \dr s$. Let us define the local martingales
\[
 L_t = \sum_{k=1}^n\int_0^t\alpha^{k,*}(s, \mu_s^*, x + B_s) \dr B^k_s \quad \text{and} \quad \mathcal{E}(L)_t = \exp\Big({L_t - \frac{1}{2}\langle L \rangle_t}\Big)
\]
Since the $\alpha^{k,*}$ are bounded, the quadratic variation of $(L_t)_{t\in[0,T]}$ is bounded and we can apply  Girsanov Theorem, which tells us that the measure $\mathbb{Q} = \mathcal{E}(L)_T \cdot \mathbb{P}$ is a probability measure on $(\Omega, \mathcal{G}_T)$ and that the processes
\[
 W^k_t = B^k_t - \int_0^t\alpha^{k,*}(s, \mu_s^*, x + B_s) \dr s = B^k_t - \langle B^k, L\rangle_t
\]
are \(n\) independent $(\mathcal{G}_t)_{t\in[0,T]}$-Brownian motions under $\mathbb{Q}$. Hence, if we define $X_t^{k,*} = x^k + B^k_t$, we have
\[
 X_t^{k,*} = x^k + \int_0^t \alpha^{k,*}(s, \mu_s^*, X_s^*) \dr s + W^k_t \quad \text{and}\quad \mu_t^* = \nu + \sum_{k=1}^n\int_0^t \delta_{X_s^{k,*}} \dr s. 
\]
Girsanov Theorem also tells us that for any bounded functional $\mathrm{F}$, we have
\begin{equation}\label{girsanov_law}
 \mathbb{E}_{\mathbb{Q}}[\mathrm{F}((X_t^*)_{t\in[0,T]})] = \mathbb{E}[\mathrm{F}((x + B_t)_{t\in[0,T]})\e^{L_T - \frac{1}{2}\langle L \rangle_T}]
\end{equation}
We now apply the Itô formula to the  $\mathrm{C}^{1,1,2}(\mathbb{R}_+\times \mathrm{E})$ function $c$. Remembering that $\alpha^{k,*}= -\nabla_{x^k} c$, we see that
\begin{align*}
 c(T, \nu + \theta_T^x, x + B_T) = & c(0, \nu, x) - \sum_{k=1}^n\int_0^T\alpha^{k,*}(s, \nu + \theta_s^x, x + B_s) \dr B^k_s \\
  & + \int_0^T[\partial_t c + \frac{1}{2}\Delta_x c + \dmu c](s, \nu + \theta_s^x, x + B_s) \dr s.
\end{align*}
Using that $c$ is a solution to \eqref{hjb_eq}, the above equality gives
\[
 L_T - \frac{1}{2}\langle L \rangle_T = c(0,\nu, x) - g(\nu + \theta_T^x, x + B_T) - \int_0^Tf(\nu + \theta_s^x, x + B_s) \dr s.
\]
Inserting this equality in \eqref{girsanov_law} completes this step.

\bigskip\noindent
\textit{Step 3:} applying again the Itô formula with the $\mathrm{C}^{1,1,2}(\mathbb{R}_+\times \mathrm{E})$ function $c$, remembering that $c$ is a solution to \eqref{hjb_eq}, we see that
\begin{align*}
 c(T,\mu_T^*, X_T^*) = & c(0, \nu, x) + \int_0^T[\partial_t c + \frac{1}{2}\Delta_x c + \dmu c](s, \mu_s^*, X_s^*) \dr s - \sum_{k=1}^n\int_0^T|\alpha^{k,*}(s, \mu_s^*, X_s^*)|^2 \dr s + M_T \\
 = & c(0, \nu, x) - \frac{1}{2}\sum_{k=1}^n\int_0^T|\alpha^{k,*}(s, \mu_s^*, X_s^*)|^2 \dr s - \int_0^Tf(\mu_s^*, X_s^*) \dr s +  M_T.
\end{align*}
where $M_t = \sum_{k=1}^n\int_0^t\alpha^{k,*}(s, \mu_s^*, X_s^*) \dr W^k_s$. Since the $\alpha^{k,*}$ are bounded, $(M_t)_{t\in[0,T]}$ is a $(\mathcal{G}_t)_{t\in[0,T]}$-martingale under $\mathbb{Q}$. Since $c(T,\mu_T^*, X_T^*) = g(\mu_T^*, X_T^*)$, we see that, passing to the expectation (under $\mathbb{Q}$), we have
\[
 c(0,\nu, x) = \mathbb{E}_{\mathbb{Q}}\Big[g(\mu_T^*, X_T^*) + \int_0^Tf(\mu_s^*, X_s^*) \dr s + \frac{1}{2}\sum_{k=1}^n\int_0^T |\alpha^{k,*}(s, \mu_s^*, X_s^*)|^2 \dr s\Big],
\]
which completes the proof.
\end{proof}

\section{Application: controlling the volume of Brownian particles}\label{sec:5}

In this section, we apply our result to control Brownian particles by the volume of their sausage. For $n$ continuous paths $(x_t^k)_{t\geq0}$ valued in $\mathbb{R}^d$, $k\in\{1, \cdots, n\}$, and a radius $\rho > 0$, the $\rho$-sausage of these paths at time $t \geq0$ is the set
\[
 \bigcup_{k=1}^n\bigcup_{s\in[0,t]}\bar{B}(x_s^k, \rho)
\]
where $\bar{B}(x_s^k, \rho)$ is the closed ball of radius $\rho$, centered in $x_s^k$. When $n = 1$ and $(x_t)_{t\geq0}$ is a Brownian path, this set is the Wiener sausage and has been widely studied, see for instance the early works of Donsker and Varadhan \cite{MR397901} or Le Gall \cite{MR942751}. It is a simple example of a non-Markovian functional of the Brownian motion and plays a key role in various stochastic phenomena, see for instance Sznitman \cite{MR1717054}. 
For recent advances on the study of Wiener sausages, one can consult \cite{schapira-sausage} and the references therein. 
For the sausage of many particles, and in particular branching Brownian particles, we refer to the work of Engländer \cite{MR1767846} and \"Oz \cite{oz_branching}.

In the following, we assume without loss of generality that $\rho =1$. If we denote by $(\mu_t)_{t\geq0} = (\sum_{k=1}^n\int_0^t\delta_{x_s^k}\dr s)_{t\geq0}$ the occupation measure of the paths, volume of the sausage can be written as
\[
 m\Big(\big\{y\in\mathbb{R}^d, \: \dr(y, \mathrm{supp}(\mu_t)) \leq 1\big\}\Big)
\]
where $m$ denotes the Lebesgue measure and we recall that for a measure $\nu \in \mathscr{M}_c^d$ the support is defined as
\[
 \mathrm{supp}(\nu) = \overline{\big\{x \in \mathbb{R}^d, \: \text{for any }\epsilon > 0, \: \nu(B(x, \epsilon)) > 0\big\}}.
\]
For a measure $\nu \in \mathscr{M}_c^d$, we also use the following notation for its associated sausage:
\[
 \sausage_\nu = \{y\in\mathbb{R}^d, \: \dr(y, \mathrm{supp}(\nu)) \leq 1\}
\]
Unfortunately, the function $\nu \mapsto m(\sausage_\nu)$ from $\mathscr{M}_c^d$ to $\mathbb{R}_+$ is not regular enough for us to apply our results. We  instead use smooth approximations of this function and solve the approximated control problems. Before stating the main result of this subsection, let us introduce some notations. Let $(f_\ell)_{\ell\in\mathbb{N}}$ be a sequence of functions from $\mathbb{R}_+$ to $\mathbb{R}_+$ that satisfies the following assumptions:
\begin{assumption}\label{assump_sequence}
 The sequence $(f_\ell)_{\ell\in\mathbb{N}}$ is a sequence of decreasing functions from $\mathbb{R}_+$ to $\mathbb{R}_+$ which are twice continuously differentiable, having the following properties:
 \begin{enumerate}[label=(\roman*)]
  \item For any $x \in[0,1)$, $f_\ell(x) \to \infty$ as $\ell\to\infty$ and for any $x \in(1,\infty)$, $f_\ell(x) \to 0$ as $\ell\to\infty$. Moreover, for any $\ell\in\mathbb{N}$, $f_\ell(1) = 1$.
  
  \item There exists $R >0$ such that for any $\ell\in\mathbb{N}$ and any $x \geq R$, $f_\ell(x) = 0$.
 \end{enumerate}
\end{assumption}
For such a sequence, we define for any $\ell\in\mathbb{N}$, the family $(f_\ell^u)_{u\in\mathbb{R}^d}$ of functions from $\mathbb{R}^d$ to $\mathbb{R}_+$ defined for any $u,v\in\mathbb{R}^d$ as $f_\ell^u(v) = f_\ell(|u-v|^2)$. Finally, let us define for any $\ell\in\mathbb{N}$, the function $g_\ell : \mathscr{M}_c^d \to \mathbb{R}_+$ such that for any $\nu \in\mathscr{M}_c^d$,
\begin{equation}\label{def_approx_saus}
 g_\ell(\nu) = -\int_{\mathbb{R}^d}\big(1 - \exp(-\langle \nu, f_\ell^u \rangle)\big) \dr u.
\end{equation}
This function is well-defined since \textit{(i)} $u \mapsto 1 - \exp(-\langle \nu, f_\ell^u \rangle)$ is measurable and \textit{(ii)} $f_\ell$ has compact support so that for any $\nu \in \mathscr{M}_c^d$, there exists some $M > 0$ such that for any $u \in B(0,M)^c$, $\mathrm{supp}(\nu) \cap \mathrm{supp}(f_\ell^u) = \emptyset$, implying that for any $u \in B(0,M)^c$, $1 - \exp(-\langle \nu, f_\ell^u \rangle) = 0$. Our main result is the following:
\begin{theorem}\label{main_thm_application}
 Let $(f_\ell)_{\ell\in\mathbb{N}}$ be a sequence of functions satisfying Assumption \ref{assump_sequence}. 
 Then for any $\ell\in\mathbb{N}$, the function $g_\ell$ defined by \eqref{def_approx_saus}, seen as a function from $\mathrm{E}$ to $\mathbb{R}$, is such that $g_\ell \in \mathrm{C}_F(\mathrm{E})$ and $\exp(-g_\ell)\in \mathrm{L}(B)$. 
 Moreover, for any $\ell\in\mathbb{N}$ and any $k\in\{1, \cdots, n\}$, let $(X_t^{k, *, \ell})_{t\in[0,T]}$ be the process from Theorem~\ref{main_thm} defined by~\eqref{optimal_process} with the functions $f = 0$ and $g_\ell$ and initial conditions $x=0$ and $\nu = 0$.
  
The sequence of processes $(X_t^{*,\ell})_{t\in[0,T]} = (X_t^{1, *,  \ell}, \cdots, X_t^{n, *, \ell})_{t\in[0,T]}$ converges in law in the space of continuous functions endowed with the uniform topology as $\ell \to \infty$ to a limiting process $(X_t^{*,\infty})_{t\in[0,T]}$ which is such that for any bounded and measurable functional $\mathrm{F}$, we have
 \[
  \mathbb{E}[\mathrm{F}((X_t^{*,\infty})_{t\in[0,T]})] = \mathcal{Z}^{-1} \mathbb{E}[\mathrm{F}((B_t)_{t\in[0,T]})\e^{m(\sausage_{\theta_T})}]
 \]
where $(B_t)_{t\in[0,T]}$ is an $n\times d$-dimensional Brownian motion, $\theta_t = \sum_{k=1}^n \int_0^t \delta_{B_s^k}\dr s$ and the normalizing constant $\mathcal{Z} = \mathbb{E}[\exp(m(\sausage_{\theta_T}))]$.
\end{theorem}


We divide the proof of this result into several intermediate results. We first prove the following result:
\begin{proposition}\label{functions_derivable}
 Let $(f_\ell)_{\ell\in\mathbb{N}}$ be a sequence of functions satisfying Assumption \ref{assump_sequence}. For any $\ell\in\mathbb{N}$, the function $g_\ell$ defined by \eqref{def_approx_saus}, seen as a function from $\mathrm{E}$ to $\mathbb{R}$, is such that $g_\ell \in \mathrm{C}_F(\mathrm{E})$ and $\exp(-g_\ell)\in \mathrm{L}(B)$.
\end{proposition}

\begin{proof}
 Throughout the proof, we fix some $\ell\in\mathbb{N}$. Let us first show that $g_\ell$ possesses a linear derivative. We first note that for any $u \in\mathbb{R}^d$, the function $\nu \mapsto 1-\exp(-\langle \nu, f_\ell^u \rangle)$ is a cylindrical function and therefore it possesses a linear derivative which is given by $(\nu, y) \mapsto \exp(-\langle \nu, f_\ell^u \rangle)f_\ell^u(y)$, see for instance Martini \cite[Example 2.11]{martini2023kolmogorov}. Hence for any $\mu,\nu \in \mathscr{M}_c^d$, we have
 \[
  g_\ell(\nu) - g_\ell(\mu) =-\int_{\mathbb{R}^d} \int_0^1\int_{\mathbb{R}^d}\exp(-\lambda\langle \nu, f_\ell^u \rangle - (1-\lambda)\langle \mu, f_\ell^u \rangle)f_\ell^u(y)(\nu - \mu)(\dr y) \dr \lambda \dr u.
 \]
Since $\mu$, $\nu$ and $f_\ell$ have compact support, we can exchange the above integrals which implies that $g$ has a linear derivative and that
\[
 \delta_\mu g_\ell(\nu)(y) = -\int_{\mathbb{R}^d}\exp(-\langle \nu, f_\ell^u \rangle)f_\ell^u(y) \dr u.
\]
Moreover, we see that for any $\nu,y\in\mathrm{E}$,  $|\delta_\mu g(\nu)(y)| \leq \int_{\mathbb{R}^d}f_\ell(|u-y|^2) \dr u = \int_{\mathbb{R}^d}f_\ell(|z|^2) \dr z$ and therefore, it is bounded.
Since $f_\ell'$ and $f_\ell''$ have compact supports, the map $y \mapsto \delta_\mu g_\ell(\nu)(y)$ is twice differentiable and we have
\[
 \nabla_y\delta_\mu g_\ell(\nu)(y) =-2 \int_{\mathbb{R}^d}(y-u)\exp(-\langle \nu, f_\ell^u \rangle)f_\ell'(|y-u|^2) \dr u
\]
and
\begin{align*}
 \nabla_y^2 \delta_\mu g_\ell(\nu)(y) = & -2\Big[\int_{\mathbb{R}^d}\exp(-\langle \nu, f_\ell^u \rangle)f_\ell'(|y-u|^2) \dr u\Big] I_d \\
  & - 4\int_{\mathbb{R}^d}(y-u)(y-u)^{T}\exp(-\langle \nu, f_\ell^u \rangle)f_\ell''(|y-u|^2) \dr u,
\end{align*}
where $I_d$ is the $d$-dimensional identity matrix. Again, for any $\nu,y\in\mathrm{E}$, we see that we have
\[
 |\nabla_y\delta_\mu g_\ell(\nu)(y)| \leq 2 \int_{\mathbb{R}^d}|y-x||f_\ell'(|y-x|^2)| \dr x \leq 2 \int_{\mathbb{R}^d}|z||f_\ell'(|z|^2)| \dr x
\]
and
\begin{align*}
 |\nabla_y^2\delta_\mu g_\ell(\nu)(y)| & \leq 2 \int_{\mathbb{R}^d}|f_\ell'(|y-u|^2)| \dr u + 4\int_{\mathbb{R}^d}|y-u|^2|f_\ell''(|y-u|^2)| \dr u \\
  & \leq \int_{\mathbb{R}^d}|f_\ell'(|z|^2)| \dr z + 4\int_{\mathbb{R}^d}|z|^2|f_\ell''(|z|^2)| \dr z,
\end{align*}
so that both derivatives are bounded. Let us now show that for any $y \in \mathbb{R}^d$, $\nu \mapsto \delta_\mu g_\ell(\nu)(y)$ possesses a linear derivative. For any $y \in \mathbb{R}^d$ and any $\mu, \in \mathscr{M}_c^d$, we have
\[
 \delta_\mu g_\ell(\nu)(y) - \delta_\mu g_\ell(\mu)(y) =\int_{\mathbb{R}^d}f_\ell^u(y) \int_0^1\int_{\mathbb{R}^d}\exp(-\lambda\langle \nu, f_\ell^u \rangle - (1-\lambda)\langle \mu, f_\ell^u \rangle)f_\ell^u(y')(\nu - \mu)(\dr y') \dr \lambda \dr u.
\]
Again, since $\mu$, $\nu$ and $f_\ell$ have compact supports, we can exchange the above integrals which implies that $\delta_\mu g_\ell$ has a linear derivative and that
\[
 \delta_{\mu \mu}^2 g_\ell(\nu)(y,y') =  \int_{\mathbb{R}^d}\exp(-\langle \nu, f_\ell^u \rangle)f_\ell^u(y)f_\ell^u(y') \dr u.
\]
Moreover, by Cauchy-Schwarz inequality, we have
\[
 |\delta_{\mu \mu}^2 g_\ell(\nu)(y,y')|^2 \leq  \int_{\mathbb{R}^d}f_\ell^u(y)^2 \dr u \int_{\mathbb{R}^d}f_\ell^u(y')^2 \dr u = \Big(\int_{\mathbb{R}^d}f_\ell(|z|^2)^2 \dr z\Big)^2,
\]
so that it is bounded. With similar arguments, one can easily show that $(y,y')\mapsto \delta_{\mu \mu}^2 g_\ell(\nu)(y,y')$ is twice differentiable and that the corresponding derivatives are bounded. This shows that $g_\ell\in\mathrm{C}_F(\mathrm{E})$. 

Let us now show that $\exp(-g_\ell)\in \mathrm{L}(B)$. We consider some $a, b, M, T >0$ and we set $\mathrm{K} = [0,T]\times \mathscr{K}_{a, b} \times [-M, M]^{nd}$. Let us show that for any $\ell \in \mathbb{N}$,
\begin{equation}
 \mathbb{E}\Big[\sup_{(t, \nu, x)\in\mathrm{K}}\exp(-g_\ell(\nu + \theta_t^x))\Big] < \infty
\end{equation}
where as usual $\theta_t^x = \sum_{k=1}^n \int_0^t \delta_{x^k + B_r^k} \dr r$, $x = (x^1, \cdots, x^n)$ and $(B_t)_{t\geq0} = (B^1_t,\cdots,B^n_t)_{t\ge 0}$ is a 
$n\times d$-dimensional Brownian motion. In the following, we first fix some $(t, \nu, x) \in \mathrm{K}$. Since for any $u\in\mathbb{R}^d$, the map $t \mapsto \langle\theta_t^x, f_\ell^u \rangle$ is non-decreasing, it comes that
\[
 \exp(-g_\ell(\nu  + \theta_t^x)) \leq \exp(-g_\ell(\nu + \theta_T^x)).
\]
Let us now introduce some notation: for any $\nu \in \mathscr{M}_c^d$ and any $R >0$, we set
\[
 \sausage_{\nu, R} = \{y \in \mathbb{R}^d, \: \dr(y, \mathrm{supp}(\nu)) \leq R\}.
\]
Consider now some $R > 0$ such that for any $M \geq R$, $f(M) = 0$. We first remark that for any $u\in\mathbb{R}^d$, we have
\[
 1 - \exp(-\langle \nu, f_\ell^u \rangle - \langle \theta_t^x, f_\ell^u \rangle) \leq \bm{1}_{\sausage_{\nu, R}\cup \sausage_{\theta_t^x, R}}(u).
\]
Indeed, if $u \notin \sausage_{\nu, R}\cup \sausage_{\theta_t^x, R}$, then $\langle \nu, f_\ell^u \rangle =  \langle \theta_t^x, f_\ell^u \rangle = 0$ so that $1 - \exp(-\langle \nu, f_\ell^u \rangle - \langle \theta_t^x, f_\ell^u \rangle) = 0$. Otherwise, it is smaller than $1$. As a consequence, we see that
\begin{equation*}
 -g_\ell(\nu + \theta_t^x) \leq m(\sausage_{\nu, R}) + m(\sausage_{\theta_t^x, R})
\end{equation*}
where $m$ denotes the Lebesgue measure. Moreover, it is straightforward to see that
\[
 m(\sausage_{\theta_t^x, R}) \leq \sum_{k=1}^n m(\sausage_{\theta_t^{x, k}, R})
\]
where for any $k\in\{1, \cdots, n\}$, $\theta_t^{x, k} = \int_0^t \delta_{x^k + B_s^k}\dr s$. By translation invariance of the Lebesgue measure, it is also clear that for any $k\in\{1, \cdots, n\}$, $m(\sausage_{\theta_t^{x, k}, R}) = m(\sausage_{\theta_t^{0, k}, R})$. Finally, since $\nu \in \mathscr{K}_{a, b}$, we have $m(\sausage_{\nu, R}) \leq (b + R)^{d}$. All in all, we showed that
\begin{equation}\label{maj_vol}
 \mathbb{E}\Big[\sup_{(t,\nu,x)\in \mathscr{K}}\e^{-g_\ell(\nu  + \theta_t^x)}\Big] \leq \e^{(b+R)^d}\mathbb{E}\Big[\exp(m(\sausage_{\theta_T^{0, 1}, R}))\Big]^n,
\end{equation}
where we used the independence between the $n$ Brownian motions. Finally, $\mathbb{E}[\exp(m(\sausage_{\theta_T^{0,1}, R}))]$ is the exponential moment of the volume of the Wiener sausage of radius $R$ and it is finite, see for instance \cite[Theorem 4.2]{MR904262} or \cite{MR1096482}.
\end{proof}

\begin{proposition}\label{conv_volume}
 Let $(f_\ell)_{\ell\in\mathbb{N}}$ be a sequence of functions satisfying Assumption \ref{assump_sequence} and let $\nu \in \mathscr{M}_c^d$. If $m(\{y \in \mathbb{R}^d, \: \dr(y, \mathrm{supp}(\nu)) = 1\}) = 0$, then $g_\ell(\nu) \to -m(\sausage_\nu)$ as $\ell\to\infty$, where the functions $(g_\ell)_{\ell\in\mathbb{N}}$ are defined by \eqref{def_approx_saus}.
\end{proposition}

\begin{proof}
 Let $u\in \{y \in \mathbb{R}^d, \: \dr(y, \mathrm{supp}(\nu)) < 1\}$ be fixed. Therefore, there exists $y \in \mathrm{supp}(\nu)$ such that $r:= |u-y| < 1$. Let $r' = (1-r) / 2 < 1$, then for any $z\in B(y,r')$, we have $|u-z| \leq (1+r) / 2 =: \bar{r} < 1$. As a consequence, we have for any $\ell\in\mathbb{N}$,
 \[
  \langle \nu, f_\ell^u \rangle = \int_{\mathbb{R}^d}f_\ell(|u-z|^2)\nu(\dr z) \geq \int_{B(y,r')}f_\ell(|u-z|^2)\nu(\dr z) \geq f_\ell(\bar{r}^2) \nu(B(y,r')),
 \]
where we used in the last inequality that $f_\ell$ is decreasing. 
Since $y \in \mathrm{supp}(\nu)$, we have $\nu(B(y,r')) > 0$ and we see that $\langle \nu, f_\ell^u \rangle \to \infty$ as $\ell\to\infty$. We showed that, if $u\in \{y \in \mathbb{R}^d, \: \dr(y, \mathrm{supp}(\nu)) < 1\}$, then $1- \exp(-\langle \nu, f_\ell^u \rangle) \to 1$ as $\ell\to\infty$.

\smallskip
Consider now some $u\in \{y \in \mathbb{R}^d, \: \dr(y, \mathrm{supp}(\nu)) > 1\}$. We have
\[
 \langle \nu, f_\ell^u \rangle = \int_{\mathrm{supp}(\nu)}f_\ell(|u-z|^2)\nu(\dr z).
\]
For any $z \in\mathrm{supp}(\nu)$, we have $|u-z| > 1$ and therefore $f_\ell(|u-z|^2) \to 0$ as $\ell\to\infty$ by assumption. Also, since $f_\ell$ is decreasing, we have $f_\ell(|u-z|^2) \leq f_\ell(1) = 1$ for any $z \in\mathrm{supp}(\nu)$ and since $\nu$ has finite mass, we can apply the dominated convergence theorem, which tells us that $\langle \nu, f_\ell^u \rangle \to0$ as $\ell \to \infty$. 

\smallskip
Since $m(\{y \in \mathbb{R}^d, \: \dr(y, \mathrm{supp}(\nu)) = 1\}) = 0$, we conclude that for a.e. $u\in\mathbb{R}^d$, we have
\[
 1- \exp(-\langle \nu, f_\ell^u \rangle) \longrightarrow\bm{1}_{\{y \in \mathbb{R}^d, \: \dr(y, \mathrm{supp}(\nu)) < 1\}}(u) \quad \text{as }\ell\to\infty.
\]
It remains to dominate $1- \exp(-\langle \nu, f_\ell^u \rangle)$ in order to apply the dominated convergence theorem. Since there exists $R >0$, such that for any $\ell\in\mathbb{N}$ and any $M \geq R$, $f_\ell(M) = 0$, we see that for any $u \in \{y \in \mathbb{R}^d, \: \dr(y, \mathrm{supp}(\nu)) > R\}$, $\langle \nu, f_\ell^u \rangle = 0$. Hence, for any $u\in\mathbb{R}^d$ and any $\ell \in\mathbb{N}$, we have
\[
 1- \exp(-\langle \nu, f_\ell^u \rangle) \leq \bm{1}_{\{y \in \mathbb{R}^d, \: \dr(y, \mathrm{supp}(\nu)) \leq R\}}(u),
\]
and the right-hand-side is integrable since $\nu$ has compact support. All in all, we get the convergence $g_\ell(\nu) \to -m(\{y \in \mathbb{R}^d, \: \dr(y, \mathrm{supp}(\nu)) <1\}) =-m(\sausage_\nu)$ as $\ell\to\infty$.
\end{proof}

In order to apply the previous result, one needs to prove that the boundary of the Wiener sausage has zero Lebesgue measure. This result is of common knowledge, nevertheless, in order to be as self-contained as possible, we give an extended proof of it.

\begin{proposition}\label{null_leb}
 For any $x\in\mathbb{R}^d$ and any $t>0$, it holds that a.s.
 \[
  m(\{y \in\mathbb{R}^d, \: \dr(y, \mathrm{supp}(\theta_t^x)) = 1\}) = 0
 \]
where $\theta_t^x = \sum_{k=1}^n\int_0^t\delta_{x^k + B_s^k} \dr s$ and 
\((B^1_t)_{t\geq0}, \cdots, (B^n_t)_{t\geq0}\)
are $n$ independent \(d\)-dimensional Brownian 
motions.
\end{proposition}

\begin{proof}
 It is enough to consider the case $n=1$ and by translation invariance, we only need to consider the case $x = 0$. For any $R > 0$, we denote by $\sausage_{t, R}$ the set $\{y \in\mathbb{R}^d, \: \dr(y, \mathrm{supp}(\theta_t^x)) \leq R\}$. Using the scale invariance of the Brownian motion, it is easy to see that
 \[
  \mathbb{E}[m(\sausage_{t, R})] = R^d \mathbb{E}[m(\sausage_{t/R, 1})].
 \]
For any $\ell\geq2$, let $a_\ell = 1 - 1/\ell$. By the monotone convergence theorem, we see that
\[
 \mathbb{E}[m(\sausage_{t, a_\ell})] \longrightarrow \mathbb{E}[m(\{y \in\mathbb{R}^d, \: \dr(y, \mathrm{supp}(\theta_t^x) < 1)\})] \quad \text{as }\ell\to\infty.
\]
On the other hand, we also see by the monotone convergence theorem that
\[
 \mathbb{E}[m(\sausage_{t/a_\ell, 1})] \longrightarrow \mathbb{E}\Big[m\Big(\bigcap_{\ell\geq2}\sausage_{t/a_\ell, 1}\Big)\Big] = \mathbb{E}[m(\sausage_{t, 1})] \quad \text{as }\ell \to\infty,
\]
where we used in the last equality that $\cap_{\ell\geq2}\sausage_{t/a_\ell, 1} = \sausage_{t, 1}$. Indeed, if $y\in \cap_{\ell\geq 2} \sausage_{t/a_\ell, 1}$, then for any $\ell\geq 2$, there exists $s_\ell \leq t /a_\ell$ such that $|y-B_{s_\ell}| \leq 1$.
Since the sequence $(s_\ell)_{\ell\geq 2}$ is bounded, we can extract a subsequence which converges to some $s$. 
Plainly, we have $s \leq t$ and by continuity of the Brownian motion, $|y-B_{s}| \leq 1$, which shows that $y \in \sausage_{t, 1}$. 
The other inclusion is straightforward. 
All in all, we conclude that
\[
 \mathbb{E}[m(\{y \in\mathbb{R}^d, \: \dr(y, \mathrm{supp}(\theta_t^x)) = 1\})] = 0 
\]
which completes the proof.
\end{proof}
We can finally give the proof of the main result of this section.

\begin{proof}[Proof of Theorem \ref{main_thm_application}]
 Let us fix some $T > 0$ and some $\ell\in\mathbb{N}$. By Proposition \ref{functions_derivable}, we can apply Theorem \ref{main_thm} with the functions $f=0$ and $g_\ell$, and consider the process $(X_t^{*, \ell})_{t\in[0,T]}$ defined by~\eqref{optimal_process} on some filtered probability space $(\Omega_\ell, \mathcal{G}_\ell, (\mathcal{G}_{t,\ell})_{t\in[0,T]}, \mathbb{Q}_\ell)$. For any continuous and bounded functional $\mathrm{F}:\mathcal{C}_T \to \mathbb{R}$, it holds that
\[
 \mathbb{E}_{\mathbb{Q}_\ell}[\mathrm{F}((X_t^{*,\ell})_{t\in[0,T]})] = \mathcal{Z}_\ell^{-1}\mathbb{E}[\mathrm{F}((B_t)_{t\in[0,T]})\e^{-g_\ell(\theta_{T})}]
\]
where $B_t =   (B^1_t,\cdots,B^n_t)$ is  a Brownian motion on \((\mathbb{R}^d)^n\), $\theta_t = \sum_{k=1}^n \int_0^t \delta_{B_s^k}\dr s$ is the occupation measure and the normalization constant is defined as $\mathcal{Z}_\ell = \mathbb{E}[\e^{-g_\ell(\theta_{T})}]$. By Propositions \ref{conv_volume} and \ref{null_leb}, it holds that a.s., $\exp(-g_\ell(\theta_T)) \to \exp(m(\sausage_{\theta_T}))$ as $\ell\to\infty$. Moreover, we get from \eqref{maj_vol} that
\[
 \mathbb{E}\Big[\e^{-g_\ell(\theta_T)}\Big] \leq \mathbb{E}\Big[\exp(m(\sausage_{\theta_T^{1}, R}))\Big]^n,
\]
where $R$ is the constant from Assumption \ref{assump_sequence}, $\theta_t^1 = \int_0^t\delta_{B_s^1}\dr s$ and
\[
 \sausage_{\theta_T^{1}, R} = \{y\in\mathbb{R}^d, \: \dr(y, \mathrm{supp}(\theta_T^1)) \leq R\}.
\]
Again, the quantity $\mathbb{E}[\exp(m(\sausage_{\theta_T^{1}, R}))]$ is finite and we can apply the dominated convergence theorem, which completes the proof.
\end{proof}

\subsection*{Acknowledgments} The authors thank Yoan Tardy for useful discussions on the topic. LB is funded by the ANR Grant NEMATIC (\href{https://anr.fr/Projet-ANR-21-CE45-0010}{ANR-21-CE45-0010}). The authors thank the research project DREAMS/NE\-MA\-TIC.

\bibliographystyle{abbrv}


\begin{thebibliography}{10}
	
	\bibitem{schapira-sausage}
	A.~Asselah, B.~Schapira, and P.~Sousi.
	\newblock Strong law of large numbers for the capacity of the {W}iener sausage
	in dimension four.
	\newblock {\em Probab. Theory Related Fields}, 173(3-4):813--858, 2019.
	
	\bibitem{benaim2002self}
	M.~Bena\"{\i}m, M.~Ledoux, and O.~Raimond.
	\newblock Self-interacting diffusions.
	\newblock {\em Probab. Theory Related Fields}, 122(1):1--41, 2002.
	
	\bibitem{benaim_2}
	M.~Bena\"{\i}m and O.~Raimond.
	\newblock Self-interacting diffusions. {II}. {C}onvergence in law.
	\newblock {\em Ann. Inst. H. Poincar\'{e} Probab. Statist.}, 39(6):1043--1055,
	2003.
	
	\bibitem{benaim_3}
	M.~Bena\"{\i}m and O.~Raimond.
	\newblock Self-interacting diffusions. {III}. {S}ymmetric interactions.
	\newblock {\em Ann. Probab.}, 33(5):1717--1759, 2005.
	
	\bibitem{benaim_4}
	M.~Bena\"{\i}m and O.~Raimond.
	\newblock Self-interacting diffusions {IV}: {R}ate of convergence.
	\newblock {\em Electron. J. Probab.}, 16:no. 66, 1815--1843, 2011.
	
	\bibitem{MR1240717}
	E.~Bolthausen.
	\newblock On the construction of the three-dimensional polymer measure.
	\newblock {\em Probab. Theory Related Fields}, 97(1-2):81--101, 1993.
	
	\bibitem{boue1998variational}
	M.~Bou\'{e} and P.~Dupuis.
	\newblock A variational representation for certain functionals of {B}rownian
	motion.
	\newblock {\em Ann. Probab.}, 26(4):1641--1659, 1998.
	
	\bibitem{budhiraja2024some}
	A.~Budhiraja.
	\newblock On {S}ome {E}xtensions of the {B}ou{\'e}-{D}upuis {V}ariational
	{F}ormula.
	\newblock \url{https://arxiv.org/pdf/2403.01562}, 2024.
	
	\bibitem{CD1}
	R.~Carmona and F.~Delarue.
	\newblock {\em Probabilistic theory of mean field games with applications.
		{I}}, volume~83 of {\em Probability Theory and Stochastic Modelling}.
	\newblock Springer, Cham, 2018.
	\newblock Mean field FBSDEs, control, and games.
	
	\bibitem{CD2}
	R.~Carmona and F.~Delarue.
	\newblock {\em Probabilistic theory of mean field games with applications.
		{II}}, volume~84 of {\em Probability Theory and Stochastic Modelling}.
	\newblock Springer, Cham, 2018.
	\newblock Mean field games with common noise and master equations.
	
	\bibitem{catellier_angelo}
	R.~Catellier, Y.~D'Angelo, and C.~Ricci.
	\newblock A mean-field approach to self-interacting networks, convergence and
	regularity.
	\newblock {\em Math. Models Methods Appl. Sci.}, 31(13):2597--2641, 2021.
	
	\bibitem{cosso2023pathdependent}
	A.~Cosso, F.~Gozzi, M.~Rosestolato, and F.~Russo.
	\newblock Path-dependent hamilton-jacobi-bellman equation: Uniqueness of
	crandall-lions viscosity solutions.
	\newblock 2023.
	
	\bibitem{MR4337713}
	A.~Cosso and F.~Russo.
	\newblock Crandall-{L}ions viscosity solutions for path-dependent {PDE}s: the
	case of heat equation.
	\newblock {\em Bernoulli}, 28(1):481--503, 2022.
	
	\bibitem{MR397901}
	M.~D. Donsker and S.~R.~S. Varadhan.
	\newblock Asymptotics for the {W}iener sausage.
	\newblock {\em Comm. Pure Appl. Math.}, 28(4):525--565, 1975.
	
	\bibitem{MR4580925}
	K.~Du, Y.~Jiang, and J.~Li.
	\newblock Empirical approximation to invariant measures for {M}c{K}ean-{V}lasov
	processes: mean-field interaction vs self-interaction.
	\newblock {\em Bernoulli}, 29(3):2492--2518, 2023.
	
	\bibitem{du2023self}
	K.~Du, Z.~Ren, F.~Suciu, and S.~Wang.
	\newblock Self-interacting approximation to mckean-vlasov long-time limit: a
	markov chain monte carlo method.
	\newblock \url{https://arxiv.org/pdf/2311.11428}, 2023.
	
	\bibitem{dupire}
	B.~Dupire.
	\newblock Functional {I}t\^{o} calculus.
	\newblock {\em Quant. Finance}, 19(5):721--729, 2019.
	
	\bibitem{MR1165516}
	R.~T. Durrett and L.~C.~G. Rogers.
	\newblock Asymptotic behavior of {B}rownian polymers.
	\newblock {\em Probab. Theory Related Fields}, 92(3):337--349, 1992.
	
	\bibitem{MR3161485}
	I.~Ekren, C.~Keller, N.~Touzi, and J.~Zhang.
	\newblock On viscosity solutions of path dependent {PDE}s.
	\newblock {\em Ann. Probab.}, 42(1):204--236, 2014.
	
	\bibitem{MR3474470}
	I.~Ekren, N.~Touzi, and J.~Zhang.
	\newblock Viscosity solutions of fully nonlinear parabolic path dependent
	{PDE}s: {P}art {I}.
	\newblock {\em Ann. Probab.}, 44(2):1212--1253, 2016.
	
	\bibitem{MR1767846}
	J.~Engl\"{a}nder.
	\newblock On the volume of the supercritical super-{B}rownian sausage
	conditioned on survival.
	\newblock {\em Stochastic Process. Appl.}, 88(2):225--243, 2000.
	
	\bibitem{fournie}
	D.~Fournie.
	\newblock {\em Functional {I}to calculus and applications}.
	\newblock ProQuest LLC, Ann Arbor, MI, 2010.
	\newblock Thesis (Ph.D.)--Columbia University.
	
	\bibitem{MR942751}
	J.-F. Le~Gall.
	\newblock Fluctuation results for the {W}iener sausage.
	\newblock {\em Ann. Probab.}, 16(3):991--1018, 1988.
	
	\bibitem{martini2023kolmogorov}
	M.~Martini.
	\newblock Kolmogorov equations on spaces of measures associated to nonlinear
	filtering processes.
	\newblock {\em Stochastic Process. Appl.}, 161:385--423, 2023.
	
	\bibitem{antonio2023controlled}
	A.~Ocello.
	\newblock Controlled superprocesses and {H}{J}{B} equation in the space of
	finite measures.
	\newblock \url{https://arxiv.org/pdf/2306.15962}, 2023.
	
	\bibitem{oz_branching}
	M.~\"{O}z.
	\newblock On the volume of the shrinking branching {B}rownian sausage.
	\newblock {\em Electron. Commun. Probab.}, 25:Paper No. 37, 12, 2020.
	
	\bibitem{pham_zang}
	T.~Pham and J.~Zhang.
	\newblock Two person zero-sum game in weak formulation and path dependent
	{B}ellman-{I}saacs equation.
	\newblock {\em SIAM J. Control Optim.}, 52(4):2090--2121, 2014.
	
	\bibitem{MR4379563}
	P.~Ren and F.-Y. Wang.
	\newblock Derivative formulas in measure on {R}iemannian manifolds.
	\newblock {\em Bull. Lond. Math. Soc.}, 53(6):1786--1800, 2021.
	
	\bibitem{sporito}
	Y.~F. Saporito.
	\newblock Stochastic control and differential games with path-dependent
	influence of controls on dynamics and running cost.
	\newblock {\em SIAM J. Control Optim.}, 57(2):1312--1327, 2019.
	
	\bibitem{stroock_var}
	D.~W. Stroock and S.~R.~S. Varadhan.
	\newblock {\em Multidimensional diffusion processes}.
	\newblock Classics in Mathematics. Springer-Verlag, Berlin, 2006.

	
	\bibitem{MR904262}
	A.-S. Sznitman.
	\newblock Some bounds and limiting results for the measure of {W}iener sausage
	of small radius associated with elliptic diffusions.
	\newblock {\em Stochastic Process. Appl.}, 25(1):1--25, 1987.
	
	\bibitem{MR1717054}
	A.-S. Sznitman.
	\newblock {\em Brownian motion, obstacles and random media}.
	\newblock Springer Monographs in Mathematics. Springer-Verlag, Berlin, 1998.
	
	\bibitem{tissotdaguette2023occupied}
	V.~Tissot-Daguette.
	\newblock Occupied processes: Going with the flow.
	\newblock  \url{https://arxiv.org/pdf/2311.07936}, 2023.
	
	\bibitem{MR4516847}
	M.~Toma\v{s}evi\'{c}, V.~Bansaye, and A.~V\'{e}ber.
	\newblock Ergodic behaviour of a multi-type growth-fragmentation process
	modelling the mycelial network of a filamentous fungus.
	\newblock {\em ESAIM Probab. Stat.}, 26:397--435, 2022.
	
	\bibitem{ustunel}
	A.~\"{U}st\"{u}nel.
	\newblock Variational calculation of {L}aplace transforms via entropy on
	{W}iener space and applications.
	\newblock {\em J. Funct. Anal.}, 267(8):3058--3083, 2014.
	
	\bibitem{MR1096482}
	M.~van~den Berg and B.~T\'{o}th.
	\newblock Exponential estimates for the {W}iener sausage.
	\newblock {\em Probab. Theory Related Fields}, 88(2):249--259, 1991.
	
	\bibitem{varadhan1969appendix}
	S.~S. Varadhan.
	\newblock Appendix to “euclidean quantum field theory” by k. symanzik,
	1969.
	
	\bibitem{MR667754}
	J.~Westwater.
	\newblock On {E}dwards' model for polymer chains. {III}. {B}orel summability.
	\newblock {\em Comm. Math. Phys.}, 84(4):459--470, 1982.
	
	\bibitem{MR853758}
	J.~Westwater.
	\newblock On {E}dwards' model for polymer chains.
	\newblock In {\em Trends and developments in the eighties ({B}ielefeld,
		1982/1983)}, pages 384--404. World Sci. Publishing, Singapore, 1985.
	
	\bibitem{willard}
	S.~Willard.
	\newblock {\em General topology}.
	\newblock Dover Publications, Inc., Mineola, NY, 2004.

	
\end{thebibliography}

{
\bigskip
\footnotesize

L.~Béthencourt, \textsc{Laboratoire J.A.Dieudonné
UMR CNRS-UNS N°7351
Université Côte d'Azur
Parc Valrose
06108 NICE Cedex 2
FRANCE}\par\nopagebreak \texttt{Loic.Bethencourt@univ-cotedazur.fr}

R.~Catellier, \textsc{Laboratoire J.A.Dieudonné
UMR CNRS-UNS N°7351
Université Côte d'Azur
Parc Valrose
06108 NICE Cedex 2
FRANCE}\par\nopagebreak \texttt{Remi.Catellier@univ-cotedazur.fr}

E.~Tanré,\par\nopagebreak \texttt{Etienne.Tanre@inria.fr}
}

\end{document}